\documentclass[11pt]{amsart}

\usepackage[utf8]{inputenc} 

\usepackage[a4paper]{geometry} 
\usepackage[pctex32]{graphics}
\usepackage{amsfonts}
\usepackage{amssymb,amscd,latexsym}
\usepackage{amsmath}
\usepackage{epsfig}
\usepackage{multirow}
\usepackage{color, soul}
\usepackage[dvipsnames,table]{xcolor}
\usepackage[most]{tcolorbox}
\usepackage{fancyvrb}   
 
\definecolor{airforceblue}{rgb}{0.2, 0.2, 0.6}
\definecolor{bleudefrance}{rgb}{0.0, 0.0, 1.0}
\definecolor{darkorchid}{rgb}{0.6, 0.2, 0.8}
\definecolor{darkorange}{rgb}{1.0, 0.55, 0.0}
\definecolor{darkspringgreen}{rgb}{0.09, 0.45, 0.27}
\definecolor{commentoutput}{rgb}{0.33,0.38,0.44}
\definecolor{output}{rgb}{0.24, 0.08, 0.08}
\definecolor{circOut}{rgb}{0.0, 0.81, 0.82}
\definecolor{Gray}{gray}{0.9}  

\usepackage[arrow,curve,matrix,tips,frame]{xy}
\usepackage{adjustbox}

\usepackage[utf8]{inputenc}
\usepackage[english]{babel} 
\usepackage{enumerate}
\usepackage{lscape}
\usepackage{hyperref} 
\usepackage{verbatim}
                 \hypersetup{ pdfborder={0 0 0}, 
                              colorlinks=true, 
                              linktoc=page,
                              pdfauthor={F. Russo and G. Staglian{\`o}}, 
                              pdftitle={On complete intersections of three quadrics in P7}}

\newtheorem{thm}{Theorem}[section]
\newtheorem{cor}[thm]{Corollary}

\newtheorem{Proposition}[thm]{Proposition}
\newtheorem{defin}[thm]{Definition}

\newtheorem{rmk}[thm]{Remark}  

\theoremstyle{remark}

\theoremstyle{definition}
\newtheorem{ex}[thm]{Example}

\newcommand{\map}{\dasharrow}

\newcommand{\SRat}{\operatorname{SRat}}

\def\p{\mathbb P}

\def\P{\mathbb{P}}
\def\I{\mathcal I}
\renewcommand{\O}{\mathcal O}
\def\rk{\operatorname{rk}}

\newcommand{\prim}{\operatorname{prim}}
\def\Sing{\operatorname{Sing}}

\newcommand{\Bl}{\operatorname{Bl}}
\newcommand{\disc}{\operatorname{disc}}
\newcommand{\alg}{\operatorname{alg}}
\newcommand{\Sec}{\operatorname{Sec}}
\newcommand{\BaseLocus}{\operatorname{BaseLocus}}

\def\Aut{\operatorname{Aut}} 
\def\Bl{\operatorname{Bl}}

\def\coker{\operatorname{coker}}

\def\Rat{\operatorname{Rat}}
\def\Pic{\operatorname{Pic}}

\def\Sing{\operatorname{Sing}}
\def\top{\operatorname{top}}

\def\NL{\operatorname{NL}}%
\def\eg{{\it e.g.}~}
\def\inv{^{-1}}

\let\phi=\varphi

\def\length{\operatorname{length}}

\numberwithin{equation}{section}

\newsavebox{\pmatrixbox}
\newenvironment{colorpmatrix}
  {\begin{lrbox}{\pmatrixbox}
   \mathsurround=0pt
   $\displaystyle
   \begin{pmatrix}}
  {\end{pmatrix}$%
   \end{lrbox}%
   \usebox{\pmatrixbox}%
   \kern-\wd\pmatrixbox
   \makebox[0pt][l]{$\left(\vphantom{\usebox{\pmatrixbox}}\right.$}%
   \kern\wd\pmatrixbox
}

\begin{document}

\title[On complete intersections of three quadrics in $\p^7$]{On complete intersections of three quadrics in $\p^7$}

\subjclass[2010]{Primary 14E08; Secondary 14M20, 14M07, 14N05, 14J28, 14J70}
\keywords{Rationality of fourfolds, Flops, Mori Theory}

\author[F. Russo]{Francesco Russo*}
\address{Dipartimento di Matematica e Informatica, Universit\` a degli Studi di Catania, 
Viale A. Doria 5, 95125 Catania, Italy
}
\email{francesco.russo@unict.it  giovanni.stagliano@unict.it}
\thanks{*Both authors were partially  supported  by the PRIN 2020 {\it Squarefree Gr\" obner degenerations, special varieties and related topics} and are members of the G.N.S.A.G.A. of INDAM}

\author[G. Staglian\` o]{Giovanni Staglian\` o*}

\begin{abstract}
We describe explicit birational maps from some rational complete intersections  of three quadrics in $\p^7$ to some prime Fano manifolds 
together with their Sarkisov decomposition via a single {\it Secant Flop}, allowing us to recover the {\it cohomologically} associated Castelnuovo
surface of general type with $K^2=2$ and $\chi=4$ (the double cover of $\mathbb P^2$ ramified along the discriminant curve of the net of quadrics defining the complete intersection) as the minimal model of the non ruled irreducible component of the base locus 
of the inverse maps. In passing we also revisit and reformulate the results in \cite{HPT2} about the existence
of infinitely many loci of  rational complete intersection of three quadrics in $\p^7$ to produce explicitly some of these
loci of  low codimension together with many other irreducible
unirational components of the Noether--Lefschetz locus. 
 \end{abstract}

\maketitle

\section*{Introduction}
Many deep contributions concerning the non stable rationality of very general Fano complete intersections of dimension at least four have been recently achieved 
(see for example \cite{Totaro2015, HPT2, Schreieder2019,NOtte}). On the other hand,  the locus $\Rat(\mathcal F)$ of geometrically rational varieties  (respectively, $\SRat(\mathcal F)$ of geometrically stably rational varieties) in an arbitrary family $\mathcal F$ of irreducible smooth projective  varieties is  the union of  countably many closed subsets by \cite{dFF} and \cite{KontsevichTschinkelInventiones} (respectively \cite{NS}). The remarkable results in \cite{HPT1, HPT2} show that there exist  (universal) families $\mathcal F$ of Fano fourfolds, whose very general element is not stably rational and such that 
$\Rat(\mathcal F)$ (and hence $\SRat(\mathcal F)$) contains infinitely many closed subsets whose union  is dense in the analytic topology. In the other cases where non stable rationality of the very general element of the universal family $\mathcal F$ has been proved, no example of a (stably) rational element in $\mathcal F$ is known so one may wonder if $\Rat(\mathcal F)$ (or  $\SRat(\mathcal F)$) is empty for these families. Smooth quartic hypersurfaces of dimension at least four are probably the most intriguing case. Even when the picture is as expected, the complete description of $\Rat(\mathcal F)$  remains a very challenging  problem, e.g. for complete intersection of three quadrics in $\p^7$. On the contrary, for the universal family $\mathcal C$ of cubic  fourfolds  there exists a precise conjecture, due to Hassett and Kuznetsov, according to which $\Rat(\mathcal C)$ is the union of countably many irreducible divisors, which is still open in its generality as well the weaker longstanding question if  $\Rat(\mathcal C)\subsetneq \mathcal C$.

The known constructions of  countably many irreducible components  of $\Rat(\mathcal F)$  for a  family $\mathcal F$ of Fano fourfolds are essentially of the same type. First one shows that a general member  $[X]\in\mathcal F$ has an elementary (Sarkisov) link to a Mori fiber space $\nu:W\to S$,
\begin{equation}\label{diagDiv}
\UseTips
 \newdir{ >}{!/-5pt/\dir{>}}
 \xymatrix{
 X'\ar[d]_\lambda\ar@{-->}[rr]^{\tilde\mu}&   &W\ar[d]^\nu                         \\
X& &S, }
\end{equation}
where: $\lambda$ is a divisorial extremal contraction; $\tilde\mu$ is a flop (maybe the identity map); $W$ is smooth and
the general fiber of $\nu$ is either a quadric surface or a del Pezzo surface of degree 6 (see \cite{HassettPlane, AHTVA, HPT2, HPT1});
$S$ is a smooth rational surface (in most cases $\p^2)$.
Then one imposes that the Mori fiber spaces $\nu:W\to S$ contains
an algebraic  2-cycle intersecting the general fiber of $\nu$ in a 0-cycle of odd degree, that is one imposes the existence of an odd degree (rational) multisection of $\nu$. An application of infinitesimal
Hodge theory as in \cite{CHM, Voisin2, HassettPlane, AHTVA, HPT2} (see also Corollary \ref{infNL} here) allows the construction of  infinitely many closed irreducible proper subsets of $\mathcal F$  of this kind, having  the right codimension $h^{3,1}(X)$. General results on quadric surfaces
or on sextic del Pezzo surfaces over arbitrary fields assure the existence of a section of $\nu$, proving the rationality of general elements in these loci and hence of every element  by \cite[Theorem 1]{KontsevichTschinkelInventiones}.

The method described above, which is very powerful but indirect by nature, does not allow the explicit determination of a (rational) section of $\nu:W\to S$ and hence of an explicit birational map from $W$ to $\p^4$ (or  to some rational fourfold), see \S \ref{birass}. Here we shall construct explicit birational maps from  general members $X\subset\p^7$ of some families in the Hilbert scheme $\mathcal H_{(2,2,2);7}$ of complete intersections of three quadrics in $\p^7$  to some rational prime Fano fourfolds $W$, whose birational representations on $\p^4$ are well known (see  Examples  \ref{esempioPiano},  \ref{esempio2}, \ref{esempio2cont}, Theorem \ref{esempioC14} and also Tables \ref{Table: rational complete intersections of three quadrics in P7} and \ref{Table: continuation of Table 1}). In this way we get a complete Sarkisov decomposition of the corresponding birational maps from $X$ to $\p^4$.  The method introduced here extends suitably the notion of  congruence of curves, introduced in \cite{RS1} and used also in \cite{Explicit, JEMS, HS} to prove  (explicit) rationality of some families of cubic or  Gushel--Mukai fourfolds. 

The birational maps $\mu:X\map W$ with $W$ a rational prime Fano manifold  have a precise description via linear systems (as in \cite{Explicit, JEMS}) and give raise to diagrams: 
\begin{equation}\label{diagDivInt}
\UseTips
 \newdir{ >}{!/-5pt/\dir{>}}
 \xymatrix{
 X'\ar[d]_\lambda\ar@{-->}[rr]^{\tilde\mu}&   &W'\ar[d]_\nu                         \\
X\ar@{-->}[rr]_{\mu}& &W }
\end{equation}
with $\lambda$ and $\nu$ extremal divisorial contractions and with $\tilde\mu$ a  flop (Sarkisov link of type $II$ according to \cite[p.~391]{HMK}). For cubic fourfolds $X\subset\p^5$ we constructed irreducible surfaces $S\subset X\subset\p^5$
such that on $X'=\Bl_S X$ there exists the flop  $\tilde\mu$ of the locus of trisecant lines to $S$ contained in $X$ (Trisecant Flop), which is expected to be an irreducible  ruled surface for  a general $S$
(recall that a general surface $S\subset\p^5$ has a two dimensional family of trisecant lines and  to be contained in a cubic through $S$ imposes a condition to a trisecant line to $S$).
Analogously, here we shall find explicit examples of maps $\mu:X\map W$  by constructing irreducible smooth surfaces $S\subset X\subset\p^7$ such that on $X'=\Bl_SX$ there exists the flop $\tilde\mu$ of the locus of secant lines
to $S$ contained in $X\subset\p^7$ (Secant Flop), which is expected to describe an irreducible ruled surface in $X$ (to be contained in $X$ imposes three conditions to the four dimensional family of secant lines to $S\subset X$: one for each quadric defining $X$), see \S \ref{s2} for definitions and  for the main results. This generalization is very natural from the point of view of the Minimal Model Program. Indeed, in the first case  $-K_{X'}=3H'-E$ so that the flop contraction is determined by the extremal ray generated by the strict transform of a (multiple of a) trisecant line to $S$ contained in $X\subset\p^5$ (such a curve $C'$ verifies  $K_{X'}\cdot C'=0$). In the second case $-K_{X'}=2H'-E$ so that the flop contraction is determined by the extremal ray generated by the strict transform of a (multiple of a) secant lines to $S$ contained in $X$. The last key step, exactly as in \cite{JEMS}, is to remark that the existence of a congruence of unisecant threefolds to $X\subset\p^7$ is equivalent to the existence of the extremal divisorial contraction $\nu:W'\to W$ with $W$ a prime Fano fourfold, see \S \ref{s2}.

 If the octic discriminant curve $C_X\subset\p^2$ of the net of quadrics defining $X\subset\p^7$  is smooth (the general case) and if  $S_X\to \p^2$ is the double cover of $\p^2$ ramified along $C_X$, then there exists an isomorphism of Hodge structures between $H^4(X,\mathbb C)$ and $H^2(S_X,\mathbb C)$
preserving the rank of algebraic cycles (see Theorem \ref{OG}). In particular,  since $\Pic(S_X)$ has rank one if $S_X$ is very general, algebraic 2-cycles on a very general  smooth $X\subset \p^7$ as above are multiples of a linear section. Hodge theory predicts that under certain hypothesis the inverse maps $\mu\inv$  of the $\mu$'s in  diagrams  \eqref{diagDivInt} contain in their base locus scheme a birational incarnation of  surfaces $U\subset W$ having the {\it same  cohomology} of $X$ (see   \S \ref{factSar} and \ref{birass} for a more detailed discussion of this principle). The same phenomenon  occurs for rational {\it special} cubic fourfolds and rational {\it special} Gushel--Mukai fourfolds with K3 surfaces as  the minimal model of the desingularizations of $U\subset W$ (see \cite{JEMS}). In the examples studied here we shall verify that  a birational incarnation of $S_X$ appears in $W$ as a component of the base locus of $\mu\inv$ (more precisely it supports the base locus of $\nu\inv$)  and that in these cases a geometric realization of  the abstract isomorphism between $H^4(X,\mathbb C)$ and $H^2(S_X,\mathbb C)$ constructed in Theorem \ref{OG} is realized via diagram \eqref{diagDivInt}. In particular, the minimal model of a desingularization of $U$ is a Castelnuovo surface of general type with $K^2=2$ and $p_g=3$ as predicted by the general discussion in \S \ref{factSar} and \ref{birass} (see Tables \ref{Table: rational complete intersections of three quadrics in P7} and \ref{Table: continuation of Table 1}). Our favorite example of {\it birational association} and of diagrams like \eqref{diagDivInt} appears in Theorem  \ref{esempioC14}, where $W$ is a (general) cubic fourfold in the Fano divisor $\mathcal C_{14}$,  $U\subset W\subset\p^5$ is a minimal Castelnuovo surface of degree 13 and sectional genus 12 while $S\subset X\subset\p^7$ is the internal projection of a general K3 surface of genus 8 and degree 14 in $\p^8$, which is explicitly {\it associated} to $W$.

In analogy with the theory developed by Hassett for {\it special} cubic fourfolds we introduce a notion of discriminant for irreducible components of the Noether--Lefschetz locus $\NL(\mathcal H_{(2,2,2);7})$, whose very general member has a rank two lattice of algebraic cycles (the expected behaviour). 
The situation here is much more intricate for many technical difficulties (see \S\S \ref{NLM24} and \ref{HCQBX}). In particular,  the discriminant may be the same for two distinct irreducible
components of the NL-locus (in this case one needs to specify also the intersection matrix on the corresponding rank two lattice). Any case, when defined, the discriminant is necessarily congruent to 0,7,12,15 modulo 16 (see Remark \ref{congdisc} and Tables  \ref{Table: unirationality} and \ref{Table: unirationality 2} for a lot of examples).

Following the approach of Fano via projection from a general line on $ X\subset\p^7$,  we also provide a direct proof that  a smooth complete intersection of three quadrics $X\subset\p^7$ is rational if it contains an algebraic 2-cycle of odd degree, essentially giving a reformulation (directly on $X$) of the results in \cite{HPT2} about  
the density of $\Rat(\mathcal H_{(2,2,2);7})$ in the Hilbert scheme  $\mathcal H_{(2,2,2);7}$, see Theorem \ref{section} and Corollary \ref{infNL}. Then we apply this result to the construction of  
many irreducible unirational components of $\Rat(\mathcal H_{(2,2,2);7})$ of codimension $2$ or $3=h^{3,1}(X)$ and of many unirational loci in  $\NL(\mathcal H_{(2,2,2);7})$ of different codimensions (see Tables \ref{Table: unirationality} and \ref{Table: unirationality 2}). We shall not insist here on the fact that our examples produce naturally also unirational irreducible components of $\NL(\mathcal M_{2,4})$ of codimension at most $3$, where
 $\mathcal M_{2,4}$ is the moduli space of surfaces of general type with $K^2=2$ and $\chi=4$ (see \S \ref{Castsurf}  for definitions).

The rationality constructions studied here, explicit or not,  
are the most  elementary ones 
from the point of view of Minimal Model Program and are  analogous to those known for cubic or  Gushel--Mukai fourfolds (see \cite{HassettPlane, AHTVA, RS1, Explicit, JEMS, HS}). 
Although $\Rat(\mathcal H_{(2,2,2);7})$ is properly contained in $\mathcal H_{(2,2,2);7}$ and although it is the union of infinitely many closed subsets, the complete description of all the components is still missing and it seems difficult to formulate a precise conjecture describing completely $\Rat(\mathcal H_{(2,2,2);7})$ (or $\SRat(\mathcal H_{(2,2,2);7}))$. The results proved here suggest that the first natural question to be settled is  if every rational smooth complete intersection of three quadrics in $\p^7$ necessarily contains an odd degree irreducible surface. 
\vskip 0.2cm

\section{Preliminaries}
\subsection{Generalities on  2-cycles on fourfolds}\label{factSar}

Let us recall some definitions.
For a smooth projective manifold  $Y$ of dimension $n=2m\geq 2$,  we say that {\it the Integral Hodge Conjecture holds for $Y$ in codimension $m$} if  
$$H^{m,m}(Y,\mathbb Z):=H^{m}(Y,\Omega_Y^m)\cap H^{2m}(Y,\mathbb Z)$$
coincides with $H^{2m}(Y,\mathbb Z)_{\alg}$, which by definition is the subgroup of $H^m(Y,\Omega^m_Y)$ generated by (classes of) algebraic cycles in $Y$ of codimension $m$.

For a smooth projective variety $Y$
of dimension $n=2m$  for which the Integral Hodge Conjecture holds in codimension $m$ (\eg for surfaces by Lefschetz Theorem;  for cubic fourfolds, see \cite{Voisin3}; for smooth complete intersections of three quadrics in $\p^7$, see section \S \ref{HCQBX}), let $$\mathcal T_Y=[H^{m}(\Omega^m_Y)\cap H^{2m}(Y,\mathbb Z)]^\perp=H^{2m}(Y,\mathbb Z)_{\alg}^\perp\subset H^{2m}(Y,\mathbb Z)$$ denote the so called {\it transcendental lattice of the middle cohomology of $Y$}.
\medskip

Since the Integral Hodge Conjecture for  two cycles holds for $\p^4$, by \cite[Lemma 15]{Voisin3} a necessary condition for the rationality of a smooth  fourfold
is that the Integral Hodge Conjecture for codimension two cycles holds on  $X$. 
\medskip

As recalled above, if $S$ is a smooth projective surface, then $H^1(S,\Omega^1_S)\cap H^2(S,\mathbb Z)=H^2(S,\mathbb Z)_{\alg}$ by Lefschetz Theorem. If $p_g(S)=0$,
then  $H^2(S,\mathbb C)=H^1(S,\Omega^1_S)$ and $\mathcal T_S=0$. Let 
$\mu:X\map W$ be a birational map with $X, W$  smooth prime Fano fourfolds. Suppose that  $\mathcal T_X\neq 0$, that
 $\mathcal T_W=0$ (for example suppose $W=\p^4$) and that we have a diagram as in \eqref{diagDivInt}. Since the blow-up of a point, of a smooth curve and  of a surface with $p_g=0$
do not affect $\mathcal T_X$ and since $\mathcal T_{X'}\simeq\mathcal T_{W'}$ for the same reason, we deduce  from $0\neq\mathcal T_X\subseteq\mathcal T_{X'}\simeq \mathcal T_{W'}$ that the base locus of $\mu\inv$ contains a surface $U\subset W$, which is necessarily irreducible (recall that  $\rho(W')=2$) and whose desingularization $\tilde U$ has  $\mathcal T_{\tilde U}\simeq \mathcal T_{W'}\simeq \mathcal T_{X'}\neq 0$. In particular $p_g(\tilde U)>0$. If moreover $\lambda:X'\to X$ is the blow-up of a point or of a curve or of a surface with $p_g=0$, then $0\neq \mathcal T_X\simeq\mathcal T_{X'}\simeq T_{W'}\simeq T_{\tilde U}$. So in many cases the trascendental cohomology determines the type of the minimal model $\overline U$ of $\tilde U$ because clearly $\mathcal T_{\overline U}\simeq \mathcal T_{\tilde U}$. 

We shall see in Examples \ref{esempioPiano}, \ref{esempio2} and Theorem \ref{esempioC14} that one is able to recover the birational type of $U$ from the diagram \eqref{diagDivInt} using the previous analysis and the next formulas.

\begin{rmk}\label{BlS}{\rm 
Let us recall that if  $Y$ is a smooth projective fourfold and if $S\subset Y$ is a smooth projective surface,  then the following isomorphisms
of Hodge structures hold for every $j=0,\ldots, 8$ (see for example \cite[Proposition 0.1.3]{Beauville1}):
\begin{equation}\label{h4split}
H^j(\Bl_S Y,\mathbb Z)\simeq H^j(Y,\mathbb Z)\oplus_{\perp}H^{j-2}(S,\mathbb Z)(-1),
\end{equation}
where $H^{j-2}(S,\mathbb Z)(-1)=H^{j-2}(S,\mathbb Z)\otimes_{\mathbb Z}\mathbb Z(-1)$ is the Tate twist, introduced to adjust
weights in order to have a morphism of Hodge structures, and where $H^l(S,\mathbb Z)=0$ for $l<0$. The homomorphism $H^j(Y,\mathbb Z)\to H^j(\Bl_SY,\mathbb Z)$ is induced by the pull-back of $\pi:\Bl_SY\to Y$. Letting 
 $E=\p(N^*_{S/Y})\to S$ (Grothendieck notation) be the exceptional divisor, the homomorphism $H^{j-2}(S,\mathbb Z)\to H^j(\Bl_SY,\mathbb Z)$ is the 
 pull-back from $S$ to $E\subset \Bl_SY$ induced by $\pi$, followed by push-forward via the inclusion of $E$.}
 \end{rmk}

Due to Theorem \ref{OG} below, one expects that the surface $\overline U$ defined above, when it exists,  is a Castelnuovo surface of general type with $K^2=2$ that we now introduce.

\subsection{Castelnuovo surfaces of general type with $K^2=2$}\label{Castsurf}

Let $C\subset\p^2$ be a smooth octic curve and let $\pi:S\to \p^2$ be the double covering of $\p^2$ branched
over $C$. Then $S$ is a simply connected minimal smooth projective surface of general type with $K_S=\pi^*(\mathcal O_{\p^2}(1))$ and $\chi(\mathcal O_S)=4$.

Castelnuovo proved that a minimal surface of general type with birational canonical map satisfies the inequality $K_S^2\geq 3p_g(S)-7$. Surfaces for which equality holds are usually  called {\it general type Castelnuovo surfaces of type $I$}. This definition has been extended to consider also the case in which $S$ is a minimal surface of general type which is a double cover of a surface of minimal degree ({\it general type Castelnuovo surfaces of type $II$}). In the particular case $K_S^2=2$, $p_g(S)=3$ and $q(S)=0$, these Castelnuovo surfaces of general type are also  the simplest examples of the so called {\it Horikawa surfaces}, i.e.  minimal surfaces of general type with $K^2=2\chi-6$, see \cite{Horikawa}. General type Castelnuovo surfaces  with $K^2=2$  have been  studied  by many authors and in different contexts, starting with the work of  Moishezon in \cite[Chapter VI]{Sha} to produce  examples of minimal surfaces of general type with $p_g=3$ such that $|3K|$ does not give a birational embedding but $|4K|$ is very ample. Indeed,  in the weighted projective space $\p(1,1,1,4)$ with variables $(x,y,z,t)$ these surfaces can be defined by  $t^2=f(x,y,z)$, where $f(x,y,z)=0$ is the equation of $C\subset\p^2=\p(1,1,1)$. This shows that the linear system $|4K_S|$ embeds $S$ in $\p^{15}$ as a surface of degree $16K_S^2=32$ and sectional genus $21$ contained in a cone over the Veronese surface $\nu_2(\p^2)\subset\p^{14}$. 

The moduli space $\mathcal M_{2,4}=\mathcal M_{K^2,\chi}$ of general type Castelnuovo surfaces with $K_S^2=2$ is an irreducible quasi-projective variety birational to the moduli space of smooth octic plane curves 
$\mathcal U\subsetneq  |H^0(\mathcal O_{\p^2}(8))|/PGL_3$ (see for example \cite{Gieseker} for the general properties of moduli spaces of surfaces of general type and \cite{Horikawa} for the connectedness in this specific case). There is an open subset $\mathcal M^0\subsetneq \mathcal M_{2,4}$, isomorphic to $\mathcal U$ and consisting of double covers of the plane ramified along  smooth octic curves. In particular, $\mathcal M^0$ is smooth and $\dim(\mathcal M_{2,4})=\dim(\mathcal U)= 44-8=36$, a fact which can also be proved directly by showing that $h^2(T_S)=0$ and $h^1(T_S)=36$ for  $[S]\in\mathcal M^0$.

Since $\chi(\mathcal O_S)=4$, Noether Formula implies $\chi_{\top}(S)=12\chi(\mathcal O_S)-K_S^2=46$. Let $[S]\in \mathcal M^0$. Since $S$ is double cover of $\p^2$ ramified along a smooth curve,  then $H^\bullet(S,\mathbb Z)$ is torsion free (see for example \cite{Cornalba}) and, being simply connected, we have $H^1(S,\mathbb Z)=0=H^3(S,\mathbb Z)$. Hence $H^2(S,\mathbb Z)\simeq \mathbb Z^{44}$, $S$ has Hodge numbers $h^{2,0}=h^{0,2}=3$, $h^{1,1}=38$ and, modulo natural identifications,  $\Pic(S)=H^1(\Omega^1_S)\cap H^2(S,\mathbb Z)$ is the subgroup of $H^2(S,\mathbb Z)$ consisting of (equivalence classes of) algebraic 1-cycles on $S$ by Lefschetz $(1,1)$-Theorem.
Moreover,  $\Pic(S)\simeq \mathbb Z\langle K_S\rangle$  for a very general $[S]\in\mathcal M^0$  (see \cite[\S 41]{Lefschetz} and also \cite[Theorem 4.4]{Ein} for a modern proof of a natural extension to weighted projective spaces). By \cite[Theorem 1]{Persson} there exist surfaces $[S]\in\mathcal M_{2,4}\setminus \mathcal M^0$ with $\rho(S)=\rk(\Pic(S))=38$, which is the maximal possible Picard number since we saw above that $ h^{1,1}(S)=38$ (no such example belonging to $\mathcal M^0$ seems to be known).

\subsection{The Noether--Lefschetz Locus of $\mathcal M_{2,4}$}\label{NLM24} 
The {\it Noether--Lefschetz Locus of $\mathcal M_{2,4}$} is by definition:
\begin{equation}\label{NLM24def}
\NL(\mathcal M_{2,4})=\{[S]\in\mathcal M_{2,4}\,:\, \rk(\Pic(S))\geq 2\}.
\end{equation} 
The locus $\NL(\mathcal M_{2,4})\cap \mathcal M^0$ contains the union of countably many irreducible subvarieties of codimension  $3=p_g(S)$ in $\mathcal M_{2,4}$, which are dense in the analytic topology on $\mathcal M^0$ (see the argument in the Introduction of \cite{CHM} for this upper bound on the codimension and also \cite[\S 5]{Voisin2} or Corollary \ref{infNL}). 
The irreducible components of $\NL(\mathcal M_{2,4})\cap \mathcal M^0$ have codimension at least two by  \cite{Esser} but there exist irreducible components of $\NL(\mathcal M_{2,4})$ having codimension one,  necessarily lying in the complement of $\mathcal M^0$. 

 In the sequel we shall explicitly construct some  irreducible components of codimension 3, whose general element is a general type Castelnuovo surface in $\mathcal M^0$ (see Examples \ref{esempioPiano}, \ref{esempio2}, Theorem \ref{esempioC14}) and whose very general element has $\Pic(S)\simeq \mathbb Z\oplus\mathbb Z$, as expected.  Although the period map is generically injective by \cite[Theorem 6.2]{Donagi}, it seems difficult to adapt the usual argument via periods, used for example for K3 surfaces in \cite[Chapter IX]{Sha}, to prove the previous expectation about $\Pic(S)$ via the period map. This essentially is also due to the fact that here $p_g(S)=3$ (and not 1 as for K3 surfaces).

An irreducible component of $\NL(\mathcal M_{2,4})\cap \mathcal M^0$ (the only cases considered in detail in the sequel) is called: {\it general} if it is of codimension  three (and hence of dimension 33) and if it is generically smooth;  {\it maximal} if it has codimension two (and hence dimension 34).  A   surface  $[S]\in\mathcal M^0$ with $\Pic(S)\simeq \mathbb Z\oplus \mathbb Z$ is called {\it very special} and these surfaces are very general elements of the general (in the above sense) components of the NL-locus studied here. 

When $S$ is a very special surface in an irreducible component of $\NL(\mathcal M_{2,4})$, we can define {\it the discriminant of $S$} as the opposite of the determinant of the intersection matrix on $\Pic(S)$. With this convention the discriminant is always positive because the determinant is always negative by Hodge Index Theorem. More generally given a rank 2 lattice $\Lambda$ with $\langle K_S\rangle\subset\Lambda\subseteq\Pic(S)$ {\it the discriminant of $\Lambda$} is the opposite
of the determinant of the intersection matrix restricted to $\Lambda$ so that it coincides with the discriminant for a very special surface in a general irreducible
component of $\NL(\mathcal M_{2,4})$.

We shall discuss in Remark \ref{congdisc} (see also the various tables below) the possible restrictions of the discriminant of very general members of  general irreducible components in 
$\NL(\mathcal M_{2,4})\cap \mathcal M^0$.

\subsection{Smooth complete intersection of three quadrics in $\p^7$}\label{nec} Let $X\subset\p^7$ be a smooth complete intersection of three quadrics. Since $N_{X/\p^7}\simeq \mathcal O_X(2)^{\oplus 3}$ and since $H^1(\mathcal O_X(2))=0$, the Hilbert scheme $\mathcal H_{(2,2,2);7}$  of smooth complete intersection of three quadrics in $\p^7$ is smooth of dimension $h^0(N_{X/\p^7})=99$ at $[X]$ and the points of $H_{(2,2,2),7}$ parametrizing smooth complete intersections are an open subset $\mathcal H\subsetneq\mathbb G(2,|H^0(\mathcal O_{\p^7}(2))|)$. The group $\Aut(X)$ is trivial for a general smooth complete intersection of three quadrics $X\subset\p^7$ and it is finite for every such complete intersection (see \cite{Benoist}). There exists a coarse moduli space $\mathcal M_{(2,2,2);7}=\mathcal H/PGL_8$ of dimension $36=99-63$, which is an affine scheme although a priori it exists  only as an algebraic space (see \cite{Benoist}, also for more general results on moduli spaces of Fano complete intersections of forms of the same degree).

The relations between a general type Castelnuovo surface with $K^2=2$ and a smooth complete intersection of three quadrics $X\subset\p^7$ are based on the fact that the discriminant curve $C_X$ of the net of quadrics defining $X$ has degree 8. Let us recall that the smoothness of $X$ does not impose any restriction on the discriminant curve (see \cite[Exemple 6.20, pg. 372]{Beauville} for examples where $C_X\subset\p^2$ is the union of eight lines) but those with smooth discriminant curve $C_X$ are an open subset  $\mathcal H^0\subsetneq \mathcal H$. Let $\widetilde{\mathcal H}$, respectively $\widetilde{\mathcal H}^0$, be the image of $\mathcal H$, respectively of $\mathcal H^0$, in $\mathcal M_{(2,2,2);7}$.

To a smooth complete intersection of three quadrics $X\subset\p^7$ with $[X]\in\widetilde{\mathcal H}$,  one can associate the double covering $S_X$ of $\p^2$ branched along $C_X$. This yields a map:
\begin{equation}\label{ratphi}
\phi:\mathcal M_{(2,2,2);7}\map\mathcal M_{2,4},
\end{equation}
defined on the open set $\widetilde{\mathcal H}$. By \cite[Remark 4.4]{Beauville},
every smooth curve $C\subset\p^2$ of degree 8 is the discriminant curve of a net of quadrics in $\p^7$, that is $\phi(\widetilde{\mathcal H}^0)=\mathcal M^0$. 

Lefschetz Theorem on Hyperplane Sections implies that  $H^\bullet(X,\mathbb Z)$ is torsion free for $[X]\in\mathcal H$. 
The intermediate cohomology $H^4(X,\mathbb C)$ of a smooth complete intersection of three quadrics  $X\subset\p^7$ has Hodge numbers $h^{3,1}=h^{1,3}=3$ and $h^{2,2}=38$ (see for example \cite[Proposition 3.2]{HTplanes}), which reminds the intermediate cohomology and the Hodge decomposition of a minimal general type Castelnuovo surface with $K^2=2$ considered above. The following result of O'Grady shows how deep the previous connection is at the level of intermediate cohomology and of integral cohomology at least for $[X]\in\mathcal H^0$.

\begin{thm}\label{OG} {\rm \cite[Theorem 0.1]{OG}} Let $X\subset\p^{2m+1}$, $m\geq 2$, be a smooth complete intersection of three quadrics. Assume that the discriminant curve  $C_X\subset\p^2$ of the net of quadrics defining $X$ is smooth. Let $\pi:S_X\to \p^2$ be the double covering branched along $C_X$ and let  $$H^{2m-2}_{\prim}(X,\mathbb C)=\coker(H^{2m-2}(\p^{2m+1},\mathbb C)\to H^{2m-2}(X,\mathbb C)).$$ There exists a morphism of Hodge structures:
\begin{equation}\label{primXS}
\psi: H^{2m-2}_{\prim}(X,\mathbb C)\to H^2_{\prim}(S_X,\mathbb C)=:\frac{H^2(S_X,\mathbb C)}{\pi^*(H^2(\p^2,\mathbb C))},
\end{equation}
which is an isomorphism of vector spaces and such that $\psi(H^{2m-2}_{\prim}(X,\mathbb Z))$ is a sublattice of
index two in $H^2_{\prim}(S_X,\mathbb Z):=\frac{H^2(S_X,\mathbb Z)}{\pi^*(H^2(\p^2,\mathbb Z))}$.

\end{thm}
\medskip
The generic Torelli Theorem for intersection of three quadrics proved in \cite[Théoréme IV.4.1]{Laszlo}
and the generic Torelli Theorem for double coverings of the plane proved in \cite[Theorem 6.2]{Donagi} together with \eqref{primXS} imply that the map $\phi$ in \eqref{ratphi} is generically injective. 

\subsection{Noether--Lefschetz Locus of $\mathcal M_{(2,2,2);7}$ for codimension two algebraic cycles}\label{HCQBX}
Let us recall that for $[X]\in\mathcal M_{(2,2,2);7}$ the Integral Hodge Conjecture for codimension two cycles holds,  that is  $H^{2,2}(X,\mathbb Z):=H^4(X,\mathbb Z)\cap H^2(\Omega^2_X)$ coincides with the subgroup generated  by  algebraic cycles. Indeed, choosing a general line $l\subset X$, every $[X]\in\mathcal H$ is birational via a flop introduced by
Beauville (see \cite{Beauville1, HPT2} and \eqref{diagB0} in \S \ref{BFLOP} below) to a smooth fourfold $W'\subset\p^2\times \p^5$, which is a quadric bundle over $\p^2$ via the restriction of the projection onto the first factor $W'\to\p^2$.
By \cite[Corollaire 8.2]{CTV} the Integral Hodge Conjecture for codimension two cycles holds for $W'$  so that it holds for $X$ by \cite[Lemma 15]{Voisin3}.
\smallskip

Since  $\Pic(S)\simeq\langle \pi^*(\mathcal O_{\p^2}(1))\rangle$ for a very general $[S]\in\mathcal M^0$ and since $\phi([X])=[S_X]\in \mathcal M^0$ for $[X]\in\widetilde{\mathcal H}^0$,   letting $h$ be the classe of a hyperplane section of $X\subset\p^7$,  \eqref{primXS} implies that  
 $$H^{2,2}(X,\mathbb Z)\simeq \mathbb Z\langle h^2\rangle$$
  for a very general $[X]\in\widetilde{\mathcal H}^0$.
The {\it Noether--Lefschetz Locus of $\mathcal M_{(2,2,2);7}$} (or of $\mathcal H_{(2,2,2); 7}$) for codimension two algebraic cycles is:
$$\NL(\mathcal M_{(2,2,2);7})=\{[X]\in\mathcal M_{(2,2,2);7}\,:\, \rk(H^{2,2}(X,\mathbb Z))\geq 2\}.$$
 Moreover,  $\NL(\mathcal M_{(2,2,2);7})\cap \widetilde{\mathcal H}^0$ contains countably many irreducible analytic subvarieties of codimension  $h^{3,1}(X)=3$ in $\widetilde{\mathcal H}^0$, whose union is a dense subset of $\mathcal M_{(2,2,2);7}$ with the euclidean topology (see Corollary \ref{infNL} below). 
 
\begin{rmk}{\rm  
By Theorem \ref{OG} and by \cite[Theorem 8.2]{Esser}, the irreducible components of the $\NL$-locus intersecting $\widetilde{\mathcal H^0}$ have codimension at least two but there  exist irreducible components of $\NL(\mathcal M_{(2,2,2);7})$, necessarily in the complement of $\widetilde{\mathcal H^0}$, having codimension one. For example the locus of smooth complete intersection of three quadrics $X\subset\p^7$ containing a smooth complete intersection of two quadrics  $T\subset\p^4\subset\p^7$ is irreducible and of  codimension one. One verifies that for such a general $X\subset\p^7$ containing the smooth surface $T$, the corresponding octic curve $C_X\subset\p^2$ has a singular point so that this irreducible component lies in the complement of $\widetilde{\mathcal H}^0$.
}
\end{rmk}

We define 
\[\mathcal M_{(2,2,2);7}^{[d,n]} := \left\{[X]\in \mathcal M_{(2,2,2);7} : \begin{tabular}{l} X \mbox{ contains some algebraic $2$-cycle $S$ and the} \\ \mbox{lattice $\langle[h^2],[S]\rangle$ has intersection matrix:} \\ \mbox{$\begin{bmatrix} (h^4)_X & h^2\cdot S \\ S\cdot h^2 & (S^2)_X \end{bmatrix} = \begin{bmatrix} 8 & d \\ d & n \end{bmatrix}$}   \end{tabular} \right\} \]
Clearly
\[\NL(\mathcal M_{(2,2,2);7}) = \bigcup_{d,n\in\mathbb{Z}} \mathcal M_{(2,2,2);7}^{[d,n]} ,\]
where some  $M_{(2,2,2),7}^{[d,n]}$ may be empty or reducible. A fourfold $[X]\in \NL(\mathcal M_{(2,2,2);7})$
has {\it discriminant $\delta$} if $[X]\in\mathcal M_{(2,2,2);7}^{[d,n]}$ with $\det \begin{pmatrix} 8 & d \\ d & n \end{pmatrix} = \delta.$ 

When $[X]$ is  very general  in an irreducible component of $\NL(\mathcal M_{(2,2,2);7})$ and when $H^{2,2}(X)=\Lambda$ is a rank two lattice with $\langle [h^2]\rangle\subset\Lambda\subseteq H^{2,2}(X,\mathbb Z)$, we can define the {\it discriminant of $\Lambda$} (and of the corresponding irreducible component
of the $\NL$ locus) as the determinant of the $2\times 2$ intersection matrix of $\Lambda$. With this definition the discriminant is always positive due to Riemann Bilinear relations.

To compute the discriminant of various rank two lattices $\Lambda=\langle[h^2], [S]\rangle$ we need to evaluate $(S^2)_X$, the self-intersection of the algebraic cycle $[S]$ in $X$. We shall do this assuming that $S\subset X\subset\p^7$ is a smooth irreducible surface, which is not restrictive because one can always find such a surface $[S]\in\Lambda$ not a multiple of $h^2$ (see the argument on at the end of page 1649 in  \cite{VoisinS}). 

\begin{Proposition}\label{selfint}
Let $S$ be a smooth irreducible surface contained 
in a smooth complete intersection 
$X\subset\mathbb{P}^{r+4}$ of type $(a_1,\ldots,a_r)$, $r\geq 1$. Then 
the self-intersection $(S)_X^2$ of $S$ in $X$ is given by 
\begin{multline}\label{formulona per S2X}
     (S)_X^2 =
\left(\frac{(r+4)(r+5)}{2} - (r+5)\sum_{1\leq i\leq r} a_i + \left(\sum_{1\leq i\leq r} a_i\right)^2
- \sum_{1\leq i<j\leq r} a_i a_j \right) H_S^2 \\ 
 + \left(r+5 - \sum_{1\leq i\leq r} a_i \right) H_S K_S + K_S^2 - c_2(T_S) ,
\end{multline}
where  $H_S$  denotes the class of a  hyperplane section of $S\subset\p^{r+4}$.
In particular, if $X\subset\mathbb{P}^7$ is a smooth complete intersection of three quadrics,
we have 
\begin{equation}\label{formulina}
     (S)_X^2 = 4 H_S^2+2 H_S K_S + K_S^2-c_2(T_S).
\end{equation}
\end{Proposition}
\begin{proof}
By the Self-Intersection Formula  $(S)_X^2=c_2(N_{S/X})$, where  $N_{S/X}$ is the normal bundle of  $S$ in $X$.
From the 
short exact sequence 
$0\to T_S \to T_X|_S \to N_{S/X} \to 0$
of tangent and normal bundles on $S$ 
we deduce 
\begin{align}
c_1(N_{S/X}) &= c_1(T_X|_S) - c_1(T_S) = c_1(T_X|_S) + K_S; \nonumber \\
\label{eq: c2NSX}
c_2(N_{S/X}) &= c_2(T_X|_S) - c_2(T_S) - c_1(T_S)\cdot c_1(N_{S/X}) \\
              &= c_2(T_X|_S) - c_2(T_S) + K_S \cdot (c_1(T_X|_S) + K_S). \nonumber 
\end{align}
Similarly, 
from the short exact sequence 
$0\to T_X \to T_{\mathbb{P}^{r+4}}|_X \to N_{X/\mathbb{P}^{r+4}} \to 0$
of tangent and normal bundles on $X$ we deduce
\begin{align}
\label{eq: c1TX}
c_1(T_X) &= c_1(T_{\mathbb{P}^{r+4}}|_X) - c_1(N_{X/\mathbb{P}^{r+4}}); \\
\label{eq: c2TX}
c_2(T_X) &= c_2(T_{\mathbb{P}^{r+4}}|_X) - c_1(T_X)\cdot c_1(N_{X/\mathbb{P}^{r+4}}) - c_2(N_{X/\mathbb{P}^{r+4}}).
\end{align}
Recalling  that
$N_{X/\mathbb{P}^{r+4}} = \mathcal{O}_X(a_1)\oplus \cdots \oplus \mathcal{O}_X(a_r)$,
we obtain 
\begin{align}
c(N_{X/\mathbb{P}^{r+4}}) &= (1+a_1 H_X)\cdots(1+a_r H_X) \nonumber \\
\label{eq: cNX}
&= 
1 + \left(\sum_{1\leq i\leq r} a_i\right) H_X + \left(\sum_{1\leq i<j\leq r} a_i a_j \right) H_X^2 + \cdots .
\end{align}
Moreover, the Euler exact sequence yields
\begin{equation}
\label{eq: cTP}
c(T_{\mathbb{P}^{r+4}}) = (1 + H_{\mathbb{P}^{r+4}})^{r+5} = 1 + (r+5) H_{\mathbb{P}^{r+4}} + 
\frac{(r+4)(r+5)}{2} H_{\mathbb{P}^{r+4}}^2 + \cdots .
\end{equation}
Now using \eqref{eq: cNX} and \eqref{eq: cTP},  relations 
\eqref{eq: c1TX} and \eqref{eq: c2TX} 
become
\begin{align}
\label{eq: c1TX-2}
c_1(T_X) &= \left(r+5 - \sum_{1\leq i\leq r} a_i \right) H_X ; \\
\label{eq: c2TX-2}
c_2(T_X) &= 
\left(\frac{(r+4)(r+5)}{2} - (r+5)\sum_{1\leq i\leq r} a_i + \left(\sum_{1\leq i\leq r} a_i\right)^2
- \sum_{1\leq i<j\leq r} a_i a_j \right) H_X^2 .
\end{align}
Finally, using \eqref{eq: c1TX-2} and \eqref{eq: c2TX-2}, we deduce \eqref{formulona per S2X} from \eqref{eq: c2NSX}.
\end{proof}

\begin{rmk}\label{congdisc}
{\rm The formula \eqref{formulona per S2X} 
can be easily applied,  
using any of the available algorithms 
to compute the topological Euler characteristic $\chi_{\top}(S)=c_2(T_S)$.
We also recall Noether Formula: $K_S^2=12\chi(\mathcal O_S) - \chi_{\top}(S)$,
which allows us to express  formula \eqref{formulina} in terms of the following four invariants of $S$: the degree $\deg(S)$, the sectional genus $g(S)$,
     $K_S^2$, and $\chi(\mathcal{O}_S)$:
     \begin{equation}\label{formulina2}
      (S)_X^2 = 2 \deg(S) + 4 g(S) + 2 K_S^2 - 12 \chi(\mathcal{O}_S) - 4 .
     \end{equation}
It follows that the discriminant $d$ of $\langle [h^2], [S]\rangle$ is:
$$d = \det \left(\begin{smallmatrix} H_X^4& H_X^2\cdot [S]\\ [S] \cdot H_X^2 & (S)_X^2 \end{smallmatrix}\right) = \det 
\left(\begin{smallmatrix} 8 & \deg(S) \\ \deg(S) & (S)_X^2 \end{smallmatrix}\right),$$
which satisfies the following condition:
\begin{equation}\label{valori ipotetici}
     d \equiv 0,7,12,15\mod 16.
\end{equation}
Tables ~\ref{Table: unirationality} and ~\ref{Table: unirationality 2} below will show  
that these values are actually possible 
provided that $d$ is small and $d\geq 16$. 
}
\end{rmk}

The following proposition allows us to estimate the codimension 
of irreducible components of the Noether--Lefschetz locus in $\mathcal M_{(2,2,2);7}$.
It follows from a standard semicontinuity argument (see \cite{Nuer-Unir} and \cite[Prop.~1.6]{S-expMath}).
\begin{Proposition}\label{conti_parametri}
Let $S\subset\mathbb{P}^7$ be a smooth irreducible surface
which is contained in a smooth complete intersection 
of three quadrics $X\subset \mathbb{P}^7$.
Assume that 
$h^1(N_{S/\mathbb{P}^7})=0$,
 $h^1(\mathcal O_S(2))=0$, and that
 $h^0(\mathcal{I}_{S/\mathbb{P}^7}(2))=h^0(\mathcal{O}_{\mathbb{P}^7}(2))-\chi(\mathcal O_S(2))$.
Then 
there is a unique irreducible component $\mathcal S\subset\mathrm{Hilb}(\mathbb{P}^7)$ 
of the Hilbert scheme 
of $\mathbb{P}^7$ that contains $[S]$, and the locus in
$\mathbb{G}(2,\mathbb{P}(H^0(\mathcal{O}_{\mathbb{P}^7}(2))))$
 of 
complete intersections of three quadrics 
 containing some surface 
of the family $\mathcal S$ has codimension at most 
\begin{equation*}
99 - \left(h^0(N_{S/\mathbb{P}^7}) + \dim\mathbb{G}(2, \mathbb{P}(H^0(\mathcal{I}_{S/\mathbb{P}^7}(2)))) - h^0(N_{S/X}) \right).
\end{equation*}
\end{Proposition}

\subsection{The blow-up of a complete intersection of three quadrics in $\p^7$ along a general line}\label{BFLOP}

Let $\pi=\pi_l:\p^7\map \p^5$ be the projection from a line $l\subset X\subset\p^7$ with $X$ an arbitrary smooth complete intersection of three quadrics. For every $r\in X\setminus l$ we have that $\overline{\pi\inv(\pi(r))}=\overline{\langle r,l\rangle\cap X\setminus l}$, that is the closure of the fiber through $r$  consists either of the plane spanned
by $r$ and $l$ or consists of line or it is a point. To see what case occurs we consider
 $\tilde\psi:\Bl_lX\to \p^5$, the resolution of the rational map $\pi$, and the equations of $\Bl_lX\subset \Bl_l\P^7$, which correspond to the data of  a $3\times 3$ matrix:
\begin{equation}\label{eqZ}
A=\left(\begin{matrix}
L&L'&L''\\
M&M'&M''\\
F&F'&F''
\end{matrix}\right ),
\end{equation} 
with $L, L', L'', M, M', M''$ linear forms on $\p^5$ and $F, F', F''$  quadratic forms on $\p^5$. Indeed, a quadric through $l$ has equation corresponding to a column of \eqref{eqZ} and the computation of the  fiber over a point $q\in\p^5$ is equivalent to solve the corresponding system with variables on the left and with $A$ evaluated in $q$. The determinant of \eqref{eqZ} defines a quartic hypersurface $Z\subset\p^5$, which is the image of $\Bl_lX$ via $\tilde\psi$ (for $q\in\p^5\setminus Z$ the fiber $\tilde\psi^{-1}(q)$ is empty, being the intersection of three linearly independent lines in $\p^2_q=\langle l, q\rangle$). If $C=\{r\in Z: \rk(A(r))\leq 1\}$, then for $q\in Z\setminus C$ the fiber $\tilde\psi^{-1}(q)$ is a point in $\p^2_q$; for $q\in C$ either the fiber $\tilde\psi^{-1}(q)$ corresponds to  a line contained in $X$ and cutting $L$ in a point (if $\rk(A(q))=1$) or, for $q$ in the locus defined by the entries of \eqref{eqZ} (equivalently
if $\rk(A(q))=0$), the fiber $\tilde\psi^{-1}(q)$ is a plane $\p^2_q$ through $l$. 
If $X\subset\p^7$ is smooth, then it contains at most a finite number of planes and it is covered by lines. In particular,  a general line $l\subset X$ is disjoint from the planes in $X$. Hence if $l\subset X$ is general,  the rank of the corresponding matrix  $A$ is positive at any point of $\p^5$ and the positive dimensional fibers of $\tilde\psi:\Bl_lX\to Z$ are isomorphic to $\p^1$. 

If $l\subset X$ is a general line, then $N_{l/X}$ is isomorphic  to $\O_{\p^1}^{\oplus 3}$.
Let $l\subset X\subset \p^7$ be a general line on an arbitrary smooth complete intersection of three quadrics and 
let $E\subset \Bl_lX$ be the exceptional divisor of $p:\Bl_lX\to X$.  Then $E\simeq \p^1\times \p^2$ since $N_{l/X}\simeq \O_{\p^1}^{\oplus 3}$ while  the restriction of $\tilde\psi$ to $E$ is given by the tautological line bundle of $\O_{\p^1}(1)^{\oplus 3}$ ($\tilde\psi$ is defined by the strict transform of hyperplanes through $l$ producing a twist by $\O_{\p^1}(1)$ in $N_{l/X}^*$). In particular, $\tilde\psi_{|E}:E\to Z\subset\p^5$ is an embedding and $\tilde E=\tilde\psi(E)\subset\p^5$ is defined by the three $2\times 2$ quadratic minors of \eqref{eqZ}. From \eqref{eqZ} we deduce that $Z$ is a quartic hypersurfaces through $\tilde E=\p^1\times\p^2$. Viceversa, a quartic hypersurface through $\p^1\times \p^2\subset\p^5 $ is the determinant of a matrix as in \eqref{eqZ} because the homogeneous ideal of $\p^1\times\p^2\subset\p^5$ is generated by the $2\times 2$ minors of a $2\times 3$ matrix of linear forms on $\p^5$ (see also \cite[\S 4 and 5]{Roth2} for the birational representation of other smooth Mukai fourfolds on special quartic hypersurfaces in $\p^5$). 

The positive dimensional fibers of $\tilde\psi$ are lines so that the exceptional locus of $\tilde\psi:\Bl_lX\to Z$ is an equidimensional surface $\tilde S$, given by the strict transforms of the lines in $X$ cutting $l$. Then $\tilde\psi(\tilde S)=C$ is an equidimensional curve, $\tilde\psi_{|\tilde S}:\tilde S\to C$ is a $\p^1$-fibration such that $\tilde C=E\cap \tilde S$ is a curve mapping isomorphically to $C$ (that is a section of $\tilde\psi_{|\tilde S}$) and the singular locus of $Z$ is supported on $C$. For a general $X\subset\p^7$ through $l$ the curve $C$ is connected
and smooth (one looks to the determinantal description of $\Sing(Z)$) so that the same is true for $\tilde S$. Without difficulty but with a careful analysis of the normal bundle of the fibres of $\tilde\psi_{\tilde S}$, one can prove that the same conclusions hold for a general line on an arbitrary smooth  $X\subset\p^7$

In conclusion,
for a general $l\subset X\subset \p^7$ with $X$ an arbitrary smooth complete intersection, the scheme $\Sing(Z)$ coincides with $C$, it is a smooth irreducible curve of degree 17 and genus 18 by Giambelli's Formulas, whose homogenous ideal is generated  by  the three degree two $2\times 2$ minors of $A$ and by the  six $2\times 2$ cubic minors of $A$. The map $\tilde\psi: \Bl_lX\to Z$ is a small birational contraction such that $\tilde\psi^{-1}(C)=\tilde S$ is a smooth irreducible surface with  $\tilde\psi:\tilde S\to C$ a $\p^1$-bundle over $C$. Let $X'=\Bl_lX\stackrel{p}\to X$. Since $-K_{X'}=2H'-2E=2(H'-E)$ with $H'=p^*(H)$ is generated by global sections and since $-K_{X}\cdot F=0$ for every fibre of $\tilde S\to C$, the surface $\tilde S$ can be flopped giving a fourfold $W'$ and a small birational contraction $\phi:W'\to Z$ such that $\tau=\phi\inv\circ\tilde\psi$ is an isomorphism in codimension one. This process can be done simply and explicitly in this case.

In $\p^2\times \p^5$ with coordinates $(y_0:y_1:y_2)\times (x_0:\ldots:x_5)$ we can consider the system of bihomogeneous equations:
\begin{equation}\label{bihom}
\left(\begin{matrix}
L&L'&L''\\
M&M'&M''\\
F&F'&F''
\end{matrix}\right )\cdot \left(\begin{matrix}
y_0\\
y_1\\
y_2
\end{matrix}\right )=\left(\begin{matrix}
0\\
0\\
0
\end{matrix}\right ).
\end{equation}  
The set of solution $W'\subset\p^2\times \p^5$ is the complete intersection of two divisors of type $(1,1)$ and of a divisors of type $(1,2)$.
The projection onto the first factor $\tilde\pi_1:W'\to \p^2$ endows $W'$ with a structure of quadric bundle over $\p^2$, whose fiber
over $(\lambda_0:\lambda_1:\lambda_2)\in\p^2$ is the quadric in $(\lambda_0:\lambda_1:\lambda_2)\times\p^5$ of equations
\begin{equation}\label{bihom2}
\left(\begin{matrix}
L&L'&L''\\
M&M'&M''\\
F&F'&F''
\end{matrix}\right )\cdot \left(\begin{matrix}
\lambda_0\\
\lambda_1\\
\lambda_2
\end{matrix}\right )=\left(\begin{matrix}
0\\
0\\
0
\end{matrix}\right ).
\end{equation}
The morphism  $\tilde\pi_2:W'\to Z\subset \p^5$ given by the restriction of the second projection to $W'$
is  birational onto its image $Z\subset\p^5$ and it induces an isomorphism between $W'\setminus\tilde\pi_2^{-1}(C)$ and $Z\setminus C$. The exceptional locus of $\tilde\pi_2:W'\to Z$ is a smooth surface $\tilde T$,
which is a $\p^1$-bundle over $C$, and $\tilde\pi_2:W'\to Z$ is a small birational contraction such that
$$W'\setminus \tilde T\simeq   Z\setminus C\simeq X\setminus \tilde S.$$

Let $\tilde E=\tilde\psi(E)\subset Z$ be the Segre 3-fold, whose homogeneous ideal is generated  by the three $2\times 2$ minors of the first two lines of $A$. Let $\alpha:Z\map \p^2$ be given by $|H^0(\mathcal I_{\tilde E}(2)|_{|Z}$ and whose general fiber $\tilde F$ is a $\p^3$ intersecting $\tilde E$ along a quadric surface $\tilde Q$. Then the  closure of a general fiber
of $\alpha$ is a quadric surface $\tilde Q'$, which from one hand is such that $\tilde F\cap Z=\tilde Q\cup \tilde Q'$ and form the other hand is the  projection of a fiber of $\tilde\pi_1:W'\to \p^2$.
So we  constructed a diagram:

\begin{equation}\label{diagB0}
\UseTips
 \newdir{ >}{!/-5pt/\dir{>}}
 \xymatrix{          
 X'=\Bl_lX\ar[d]_{p} \ar@{-->}[rr]^{\tau}\ar[rd]_{\tilde\psi}&                             &W'\ar[ld]^{\tilde\pi_2}\ar[d]^{\tilde\pi_1}\\
X\ar@{-->}[r]_{\pi_l}&Z\ar@{-->}[r]_{\alpha}&\p^2.
}
\end{equation}
The map $\tau=\tilde\pi_2^{-1}\circ\tilde\psi$ is the flop of the surface $\tilde S$ (and $\tau^{-1}$ is the flop of $\tilde T$) and the external part of the diagram is an elementary link of type $I$ according to \cite[pg. 391]{HMK} (here $\tilde\pi_1;W'\to \p^2$ is an elementary contraction endowing $W'$ with a structure
of Mori fiber space over $\p^2$).

The top part of  diagram  \eqref{diagB0}
has been constructed in a little bit  different way by Beauville in \cite[Exemple 1.14]{Beauville1} so that the map $\tau$ will be dubbed as the {\it Beauville Flop} although flops were
defined and studied some years after the appearance of \cite{Beauville1}. This is also the first instance of the flops studied by many authors in different contexts  and considered also in \cite{JEMS}. 

The rational map 
$$\beta=\alpha\circ \pi_l:X\map\p^2$$
was studied by Fano in \cite[\S 3]{FanoBer} (see also \cite[\S 2]{Roth}) and $\beta$ (or $\alpha$) will be dubbed as the {\it Fano map of $X\subset\p^7$} when $l\subset X$ is a general line.

 \subsection{Rationality versus  stable non rationality of smooth complete intersections of three quadrics in $\p^7$}\label{RatM2227}

Let us introduce  some definitions and recall some recent results about the rationality and stable non rationality of smooth complete intersections of three quadrics in $\p^7$. Let
\begin{equation}\label{RatC}
\Rat(\mathcal M_{(2,2,2);7})=\{[X]\in\mathcal M_{(2,2,2);7}\;:\; X\text{ is rational }\}\subseteq\mathcal M_{(2,2,2);7}.
\end{equation}

By results of de Fernex and Fusi (see \cite{dFF}) the locus $\Rat(\mathcal M_{(2,2,2);7})$ is the union of at most countably many constructible sets (in the Zariski topology). By a result of  Kontsevich and Tschinkel  (see \cite[Theorem 1]{KontsevichTschinkelInventiones}) the set  $\Rat(\mathcal M_{(2,2,2);7})$ is stable under specialization so that $\Rat(\mathcal M_{(2,2,2);7})$ is the union of at most countably many irreducible closed algebraic subsets. 

The main results of Hassett, Pirutka and Tschinkel in \cite{HPT2} assure that the very general $[X]\in\mathcal M_{(2,2,2);7}$ is not stably rational and that there exist countably many irreducible subvarieties of $\Rat(\mathcal M_{(2,2,2);7})$ of codimension $3=h^{3,1}(X)$, whose union is  dense in $\mathcal M_{(2,2,2);7}$ in the analytic topology and which are irreducible components of  $\NL(\mathcal M_{(2,2,2);7}).$

 The countably many loci  in $\Rat(\mathcal M_{(2,2,2);7})$ produced in  \cite{HPT2} have not an explicit description  and their existence  is based on non constructive techniques exhibiting  sections of the associated quadric bundles over $\p^2$ obtained via the Beauville flop. 
 \medskip
 
The Fano Map allows us to  provide a little bit more explicit analysis of the infinitely many components in $\Rat(\mathcal M_{(2,2,2);7})$ constructed in \cite{HPT2} 
and revisited here to provide concrete examples (in \cite{HPT2} the authors worked on the image $W'$ of $X$ via a Beauville Flop
and not directly on $X$).

\begin{thm}\label{section} Every smooth complete intersection of three quadrics in $\mathbb{P}^7$ 
containing a $2$-cycle of odd degree is rational.
\end{thm}
\begin{proof}  Let  $S\in H^{2,2}(X,\mathbb Z)$ be (the class of) an algebraic $2$-cycle of odd degree $S\cdot h^2=2m+1$, whose support consists of a finite number of irreducible surfaces $S_1,\ldots, S_m$. 
Since $h^4=\deg(X)=8$, we can also suppose $S\not\in \langle h^2\rangle$.
Let  $S=\sum_{i=1}^ma_i S_i$ with $ a_i\in \mathbb Z\setminus 0$. From  $2m+1=\sum_{i=1}^ma_i \deg(S_i)$, we deduce the existence of at least one $i\in\{1,\ldots, m\}$ such that $\deg(S_i)$ is
odd. So we can suppose $S=S_i$. 
There exists a general line $l$ such that $l\cap S=\emptyset$. Let $\pi_l:X\map Z\subset\p^5$ be the projection from $l$, 
let 
 $\tilde E\subset Z\subset\p^5$ be the Segre 3-fold $\p^1\times\p^2$ determined by $l$ and let $\alpha:Z\map \p^2$ the restriction to $Z$ of the map $\psi:\p^5\map \p^2$ defined by $|H^0(\mathcal I_{\tilde E}(2))|$. The closure of every fiber $\psi^{-1}(q)$ is a linear space $\p^3_q$ cutting $\tilde E$ along a smooth quadric $\tilde Q_q$ and $Z$ along  the quartic surface $Q_q\cup \tilde Q_q$. Then the closure of the fiber of $\alpha^{-1}(q)$ is the quadric $Q_q$ which is smooth for $q$ general, showing that $\alpha$ is a rational fibration in quadric surfaces whose general member is smooth. Then $S_q=\pi_l^{-1}(Q_q)=\pi_l^{-1}(\p^3_q)\subset X$ is an irreducible surface containing $l$ and whose class is $h^2$. Then $2m+1=\deg(S)=S_q\cdot S$. Since  $S\cap l=\emptyset$, then $\deg(S)=\deg(\pi_l(S))=2m+1$ and  $2m+1=\deg(\p^3_q\cap \pi_l(S))=\deg((Q_q\cup \tilde Q_q)\cap \pi_l(S))=\deg(Q_q\cap \pi_l(S))$ because $\pi_l(S)\cap \tilde Q_q=\emptyset$ (recall that $\pi_l^{-1}(\tilde E)=l$). By Springer Theorem (see \cite[p.~32]{Levico}) the generic fiber of $\alpha:Z\map\p^2$ is a smooth quadric surface over $K(\p^2)$ with  a point in $K(\p^2)$ and hence rational over $K(\p^2)$. Then $Z$ is a rational variety birational to $X$, concluding the proof.
\end{proof}

As we saw above, a very general smooth complete intersection of three quadrics does not contain any algebraic surface of odd degree and neither algebraic 2-cycles of odd degree. Those containing such 2-cycles necessary belong to $\NL(\mathcal M_{(2,2,2);7})$.

\begin{ex}\label{NLP}
To contain a plane imposes 6 conditions to a quadric in $\p^7$. Hence to contain a plane imposes 18 conditions to a complete
intersection of three quadrics, which are independent as one verifies in an explicit example. Since $\dim(\mathbb G(2,7))=15$ and since a general smooth complete intersection of three quadrics through a fixed plane $P$ contains only the plane $P$, we deduce the well known fact that
 $$\NL(\mathcal P)=\{[X]\in \mathcal M_{(2,2,2);7}\,:\, P\subset X\}\subset \mathcal M_{(2,2,2);7}$$
 is irreducible, has codimension 3=18-15 and is generically smooth. Moreover, $\NL(\mathcal P)$ is an irreducible component of codimension 3 of $\Rat(\mathcal M_{(2,2,2);7})$ by Theorem \ref{section} or simply by projecting from $P$ onto $\p^4$ a $X$ through $P$ (see Example \ref{esempioPiano} where we shall also prove that  $H^{2,2}(X,\mathbb Z)\simeq\langle h^2, P\rangle$ for a very general $[X]\in \NL(\mathcal P)$ and that the discriminant of the fourfolds in $\NL(\mathcal P)$ is 31).
 \end{ex}
\medskip

A very powerful technique, introduced by Green (see \cite[Section 5]{CHM}) and later developed by Voisin in \cite[\S 5.3.3, \S 5.3.4]{Voisin2}, can be applied to deduce some strong  consequences on $\Rat(\mathcal M_{(2,2,2);7})$ (see also \cite{HassettPlane} and  \cite{HPT1,HPT2}).

\begin{cor}\label{infNL} The union of the loci 
$\mathcal M_{(2,2,2);7}^{[2m+1,n]}$, with $n,m\in\mathbb{N}$, 
is 
dense in the analytic topology on $\mathcal M_{(2,2,2);7}$ and it contains infinitely many irreducible components of $\Rat(\mathcal M_{(2,2,2);7})$ of codimension  three in $\mathcal M_{(2,2,2);7}$.
\end{cor}
\begin{proof} By definition 
$$\NL(\mathcal M_{(2,2,2);7})=\{[X]\in\mathcal M_{(2,2,2);7}\,:\, \exists\, S\in H^{2.2}(X,\mathbb Z)\setminus \langle  h^2\rangle\}.$$
If $S\cdot h^2=2m+1$ for some $m\in\mathbb Z$, then $S\not\in\langle h^2\rangle$ because $h^4=8$.
Since there exists the irreducible component $\NL(\mathcal P)$ of codimension $3=h^{1,3}(X)$ of $\mathcal M_{(2,2,2);7}$,
Green's argument  (see \cite[\S 5]{CHM} and \cite[\S 5.3.3, \S 5.3.4]{Voisin2}) shows that
for every $[X]\in\mathcal H\subset G(3,H^0(\O_{\p^7}(2)))$ there exists a simply connected neighbourhood $U$ of $[X]$ in $\mathcal H$ such that $H^4(X',\mathbb R)$ is naturally identified with $H^4(X,\mathbb R)$ for every $[X']\in U$ due to Ehresmann's Theorem and such that the image of the map $$G_U:U\times H^{2,2}_{\mathbb R}\to  H^4(X,\mathbb R),$$
introduced in \cite[Section 5]{CHM}, contains an open subset $\tilde U$. Since $[X']\times H^{2,2}_\mathbb R=H^2(\Omega^2_X)\cap H^4(X,\mathbb R)$, then $\NL(\mathcal M_{(2,2,2);7})\cap U$ is the projection on $U$ of the preimage $Q_U$ via $G_U$ of the points in $H^4(X,\mathbb Q)$ not lying in $\langle h^2\rangle$. Since this set is dense in $H^4(X,\mathbb R)$ in the analytic topology, the set $Q_U$ is dense in $U\times H^{2,2}_\mathbb R$ in the analytic  topology and the same is true for its image $\NL(\mathcal M_{(2,2,2);7})\cap U$ in $U$.

So  $\NL(\mathcal M_{(2,2,2);7})$ is dense in the euclidean topology on $\mathcal M_{(2,2,2);7}$ and the previous argument also shows the existence of countably many irreducible components of codimension three of $\NL(\mathcal M_{(2,2,2);7})$. Since the Hodge Integral Conjecture holds for $[X]\in \mathcal M_{(2,2,2);7}$ and since the density of these components of $\NL(\mathcal M_{(2,2,2);7})$ is essentially equivalent to the density of $H^4(X,\mathbb Q)_{\prim}$ in $H^4(X,\mathbb R)_{\prim}$, to conclude it is enough to remark  that the rational multiples of $S\in (h^2)^\perp \subset H^4(X,\mathbb Z)$ such that $S\cdot h^2=2m+1$, $m\in\mathbb Z$, are also dense in $(h^2)^\perp \subset H^4(X,\mathbb R)$.  
\end{proof}

As an application of the previous results we collected in  Tables \ref{Table: unirationality} and \ref{Table: unirationality 2} at the end of the paper  some irreducible components of $\NL(\mathcal M_{(2,2,2);7})$, whose codimension is at most three. Those highlighted in grey corresponds to odd degree surfaces and hence to loci of $\Rat(\mathcal M_{(2,2,2);7})$, mostly of codimension 3.

It would be very interesting to provide a complete description of  $\Rat(\mathcal M_{(2,2,2);7})$ inside $\mathcal M_{(2,2,2);7}$  starting from the first
natural question if 
$$\bigcup_{m,n\in\mathbb{N}} \mathcal M_{(2,2,2);7}^{[2m+1,n]}\subsetneq \Rat(\mathcal M_{(2,2,2);7}).$$ For cubics fourfolds containing a plane, which are naturally quadric bundles over $\p^2$ via the projection from the plane, the existence of a section is  not a necessary condition for rationality so one suspects that the previous inclusion might be strict.

\subsection{The birationally associated Castelnuovo surface of general type and unisecant 3-folds to a rational complete intersection of three quadrics in $\p^7$}\label{birass}

The existence of a (rational) section of $\alpha:Z\map \p^2$ (or of $\beta:X\map \p^2$) determines a map from the generic fiber of $\alpha$ and $\p^2_{K(\p^2)}$, yielding  a birational map  $Z\map \p^2\times\p^2$ and hence a birational map $\gamma:X\map \p^2\times\p^2$.

One expects to find some birational incarnation in $\p^2\times\p^2$ of the
{\it cohomologically associated} Castelnuovo surface $S_X$. Indeed, considering a Beauville Flop of a general line in $X$ we get a quadric fibration $\pi:W'\to \p^2$ birational to $X$ and with a rational section $s:\p^2\map W'$. Let $\tilde T=\overline{s(\p^2)}\subset W'$. The birational map $W'\map \p^2\times\p^2$ is the projection from $\tilde T$ (or from the corresponding $T\map \p^2$ on $X$). Since the discriminant curve of $\pi:W'\to\p^2$ is a smooth octic plane curve $C\subset\p^2$ and since the projection from $\tilde T$ contracts the lines in each fiber intersecting $\tilde T$, we find a surface $U\subset\p^2\times\p^2$, birational to a degree two cover of $\p^2$ ramified along $C$ and contained in the base locus of the birational map $\p^2\times\p^2\map X$.
The determination of $U$, which is singular by the Double Point Formula, and of its non singular minimal model $\overline U$ are intricated, the situation being analogous to the case of  cubic fourfolds containing a plane and admitting a section of the corresponding fibration (in this case $U\subset\p^2\times\p^2$ is a singular K3 surface with a single well known exception, see \cite[\S 5]{HassettPlane}).  Moreover, the existence of the section is assured by Springer Theorem so that it is far from being explicit and  very hard to find also in simple cases. So although one knows a priori that some loci of complete intersection of three quadrics in $\p^7$ are rational
via a map not defined along a birational incarnation of the cohomologically associated Castelnuovo surface, in practice it is very difficult to describe explicitly the maps $\gamma$ and $\gamma^{-1}$ (base locus, factorization via elementary Sarkisov links, etc, etc).
 \smallskip
 
 Then one looks for a different  (possibly algorithmic) approach to construct birational maps from $X\subset\p^7$ to some rational prime fourfold and to describe their Sarksov factorization.
 We start with a general remark.  If there exists a birational map $\phi:X\map \p^4$, we can extend it to a map $\tilde\phi:\p^7\map \p^4$.
 
 Let $y\in \p^4$ be a general point and let $F_y=\tilde\phi\inv(y)\subset \p^7\setminus \BaseLocus (\tilde\phi)$. By definition $\dim(F_y)=3$ and $F_y\cap X$ is a point because $\phi$ is birational. Then it is not difficult to see that $\overline{F_y}$ is irreducible for $y\in \p^4$ general, yielding a $4$-dimensional family $\mathcal W$ of irreducible 3-folds such that through a general point of $\p^7$ there passes a unique member of the family and such that the general member of the family intersects $X$ in a unique point outside $\BaseLocus(\tilde\phi)$. In other words, we get a diagram:
 \begin{equation}\label{diagramma congruenza Fy}
\xymatrix{
 \mathcal V\ar[dr]^{p}\ar[d]_{\pi}&\\
\mathcal W&\p^7,}
\end{equation}
where $\pi:\mathcal V\to \mathcal W$
is the  universal family over $\mathcal W$ and where  $p:\mathcal V\to \p^N$ is  the tautological morphism, which is  birational by hypothesis.  
Moreover $\pi\circ p\inv_{|X}: X\map \mathcal W$ is birational so that $\tilde X=\overline{p\inv(X)}$ defines a rational section $s:\mathcal W\map \tilde X$ of $\pi$.

Viceversa, suppose we have a diagram like \eqref{diagramma congruenza Fy} such that for a general point $x\in X$ there is a unique variety $F_x$ of $\mathcal W$ passing through
$x$ and cutting $X$ outside the base locus of $p\inv$ only in $x$. In this case we  call $\mathcal W$ (or $\pi:\mathcal V\to \mathcal W$) {\it a congruence of unisecant $3$-folds  to $X\subset \p^7$}.

Let $\psi:X\map\mathcal W$ be the map defined by $\psi(x)=[\tilde F_x]$, which is well defined by the previous
hypothesis. Then $\psi$ is birational and one may construct some birational incarnation of $\psi$, hopefully with $\mathcal W$ equal to (or replaced birationally by) a prime Fano fourfold $W$. 
In conclusion: a rational smooth complete intersection of three quadrics $X\subset\p^7$  admits a congruence of  unisecant 3-folds  with $\mathcal W$ birational to $X$ (and in particular rational); a smooth complete intersection of three quadrics  $X\subset\p^7$ admitting a congruence of unisecant 3-folds  is birational to the parameter space $\mathcal W$ of the congruence.

This a generalization of the definitions used in \cite{RS1, Explicit, JEMS}, where we considered congruences of unisecant irreducible curves
to cubics or  Gushel--Mukai fourfolds. 
\medskip

Suppose now we have a rational complete intersection of three quadrics $X\subset\p^7$ and the associated  congruence of unisecant irreducible 3-folds to $X$, whose base locus (that is the base locus of $p\inv$) is indicated by $B$. If $V\subset\p^7$ is a general complete intersection of two quadrics vanishing on $X$, then general unisecant irreducible 3-folds of the congruence cut $V\subset\p^7$ along 
an irreducible curve which is unisecant to $X$ (we restrict everything to $V$), that is the congruence of unisecant 3-folds to $X$ and $V$ determine  a four dimensional family of curves in $V$
which are unisecant to $X$ (outside $B\cap X$) and such that through a general point of $V$ there passes a unique curve of the family. In particular the congruence
of unisecant 3-folds to $X$ determines an eight dimensional family of curves of degree $e\geq 1$ which are $(2e-1)$-secant to $B\cap X$ and such that the locus
of these curves through a general point of $\p^7$ is the unique 3-fold of the congruence passing through the point. If we take a general quadric $Q$ containing $X$, then through a general point $y\in Q$ there passes a unique surface $S_y=\overline{F_y}\cap Q$ such that through a general point of $S_y$ there passes a one dimensional family of $(2e-1)$-secant curves to $B$
and unisecant $X$, whose union is  $S_y$.

So we have essentially proved that  the rationality of $X\subset\p^7$ determines and is (in a certain sense) determined by a  huge family  of rational curves  of degree $e$ in $\p^7$ which are $(2e-1)$-secant to a fixed $S\subset X$ (not necessarily a surface, a priori).

 In the examples of rational complete intersections $X\subset\p^7$ studied in the sequel we shall find a surface $S\subset X\subset \p^7$ such that the 3-folds of the associated congruence cut $X$ only along $S$ and such that $\pi\circ p\inv$ can be realized (birationally) as a map $\phi:\p^7\map W$ (with $W$ a prime Fano manifold) via a linear system having the right multiplicities along $S$ (see  Examples \ref{esempioPiano},  \ref{esempio2} and  Theorem \ref{esempioC14}). Then we shall show  that in these examples there is a {\it  Secant Flop} connecting $X$ and the resulting prime manifold $W$.
 
 Moreover, we shall   describe also the inverse maps $\phi\inv:W\map X$  given by  linear systems whose base loci contain a (not necessarily minimal) general type Castelnuovo surface $U$, whose minimal model is isomorphic to $S_X$.  From one hand the surface $U$ parametrizes  (birationally) the (2e-1)-secant curves  of the congruence  contained in $X$ and on the other hand it realizes the (trascendental) cohomology of $X$ in $W$ (or in $W'$). 
\section{Congruences of unisecant 3-folds to some rational $X\subset\p^7$}\label{Crat}

We  provide a detailed description of two examples  of  congruences of unisecant 3-folds to some general members of some irreducible components of  $\Rat(\mathcal M_{(2,2,2);7})$, starting from the well known case of $\NL(\mathcal P)$. The fourfolds in $\NL(\mathcal P)$ have been studied by  many classical authors for their central role in rationality questions (see for example \cite[\S 2]{Roth} and also \cite[\S 3]{Roth2} for an application to rationality of prime Mukai fourfolds of genus $g>6$) and
have been revisited  in modern terms many times (e.g. \cite[Example 1, pg. 595]{AS} or \cite[pg. 620]{HTplanes}). Our presentation here will focus on the congruence of unisecant $\p^3$'s (and on the associated eight dimensional family of unisecant lines to $P$) naturally appearing and also on the birationally {\it associated} general type Castelnuovo surface with $K^2=2$. This well known analysis will serve as a natural prototype  for more intricate and sophisticated examples investigated in the sequel.

\begin{ex}\label{esempioPiano} Let $P\subset\p^7$ a plane and let $\pi:\p^7\map \p^4$ be the projection from $P$. For every $r\in\p^7\setminus P$ we have $\overline{\pi\inv(\pi(r))}=\langle r,P\rangle=\p^3_r$.  The restriction $\tilde\psi$ of this map to a general $X\subset\p^7$ is birational, as is well known, and $\pi$ induces the congruence consisting of the  $\p^3_r$, $r\in\p^4$ general, which cut $X$ in a point outside $P$. The curves determined by the congruence are nothing but the lines in $\p^7$ cutting  $P$, which have dimension eight and are parametrized by $\mathbb G(1,\p^3)$-bundle over $\p^4$. What is  interesting for the sequel is the determination of the {\it cohomologically associated} $S_X$  
as the base locus scheme of $\tilde\psi\inv$.

 We can solve the map and obtain the diagram:
\begin{equation}\label{GammaP}
\UseTips
 \newdir{ >}{!/-5pt/\dir{>}}
 \xymatrix{
 &\overline{\Gamma_\pi}\ar[ld]_p\ar[rd]^{\tilde\pi}&                              \\
\p^7\ar@{-->}[rr]_{\pi}& &\p^4, 
}
\end{equation}
where $p:\overline{\Gamma_\pi}\to \p^7$ is the blow-up of $\p^7$ along $P$ while
$\tilde\pi:\overline{\Gamma_\pi}\to\p^4$ is the projectivization of the rank four bundle $\mathcal E\simeq\mathcal O_{\p^4}^{\oplus 3}\oplus\mathcal O_{\p^4}(1)$.

Reasoning as in \S \ref{BFLOP} we deduce that the  equations of $\Bl_PX\subset \Bl_P\P^7$ correspond to the data of  a $3\times 4$ matrix:
\begin{equation}\label{eqU}
A=\left(\begin{matrix}
L_{1,1}&L_{1,2}&L_{1,3}\\
L_{2,1}&L_{2,2}&L_{2,3}\\
L_{3,1}&L_{3,2}&L_{3,3}\\
F_1&F_2&F_3
\end{matrix}\right ),
\end{equation} 
with $L_{i,j}$'s linear forms on $\p^4$ and $F_k$'s  quadratic forms on $\p^4$. Let $\tilde\psi:\Bl_PX\to \p^4$ be the restriction of $\tilde\pi$. The $3\times 3$ minors of \eqref{eqU} define a scheme $U\subset\p^4$ of dimension 2, which is irreducible since $\Pic(\Bl_PX)\simeq\mathbb Z\oplus\mathbb Z$. For $q\in\p^4\setminus U$ we have that $\tilde\psi^{-1}(q)$ is a point, being the intersection of three linearly independent planes in $\p^3_q$ (see also \S \ref{BFLOP} for a similar computation); for $q\in U\setminus B=\{r\in U: \rk(A(r))=1\}$ the fiber $\tilde\psi^{-1}(q)$ is a line in $\p^3_q$ cutting $P$; for $q\in B$ the fiber $\tilde\psi^{-1}(q)$ corresponds to  a plane contained in $X$ and cutting $P$ in a line.

For a smooth $X\subset\p^7$ we have at most a finite number of  planes and hence at most a finite number of planes of the form $\tilde\psi\inv(q)$. For a general $X$ through $P$, there no planes of this form. Hence, for  a general $X\subset\p^7$ through $P$, the base locus scheme $U\subset\p^4$ of $\tilde\psi\inv$ is a smooth irreducible surface, whose ideal is generated by a cubic form and three quartic forms. The linear system $|H^0(\mathcal I_U(4))|=\p^7$ defines $\tilde\psi^{-1}:\p^4\map X\subset\p^7$ and exhibits $\Bl_PX$ as $\Bl_U\p^4$. Since the intersection of a general $\p^5$ through $P$ cuts $X$ along a degree 8 surface containing $P$, i.e. along a degree 7 surface outside $P$, we deduce $\deg(U)=16-7=9.$

 The diagram
\begin{equation}\label{BlPX=BlUP4}
\UseTips
 \newdir{ >}{!/-5pt/\dir{>}}
 \xymatrix{
 &\Bl_PX=\Bl_U\p^4\ar[ld]_p\ar[rd]^{\tilde\pi}&                              \\
X\ar@{-->}[rr]_{\pi}& &\p^4
}
\end{equation}
induces an isomorphism of Hodge structures (see \eqref{h4split} for notation and details): 
\begin{equation}\label{isoPXU}
H^j(X,\mathbb Z)\oplus_\perp H^{j-2}(P,\mathbb Z)(-1)\simeq H^j(\p^4,\mathbb Z)\oplus_\perp H^{j-2}(U,\mathbb Z)(-1).
\end{equation}
The isomorphism in \eqref{isoPXU} restricts to an isomorphism of Hodge structures between $H^4(X,\mathbb Z)$ and $H^2(U,\mathbb Z)(-1)$, yielding in particular 
$\rk(H^{2,2}(X,\mathbb Z))=\rk(\Pic(U))$. From \eqref{isoPXU} we can compute the invariants of $U$ directly: $p_g(U)=h^{2,0}(U)=h^{3,1}(X)=3$, $q(U)=h^{1,0}(U)=h^{2,1}(\Bl_U\p^4)=h^{2,1}(\Bl_PX)=h^{2,1}(X)=0$ (yielding $b_1(U)=0$), $\chi(\mathcal O_U)=4$, $\chi_{\top}(U)=2+b_2(U)=2+b_4(X)=46$, $K^2_U=12\chi(\mathcal O_U)-\chi_{\top}(U)=2$. Since $p_g(U)=3$ and since $K_U^2=2>0$, then $U$ is a surface of general type such that $K^2=3p_g-7$. To deduce that  $U\subset\p^4$ is a minimal Castelnuovo surface of general type with $K^2=2$ is sufficient to remark that there are isomorphism of Hodge structures $H^2(U,\mathbb C)_{\prim}\simeq H^4(X,\mathbb C)_{\prim}\simeq H^2(S_X,\mathbb C)_{\prim},$ where the last isomorphism holds because for a general $X$ through $P$ the octic discriminant curve is smooth. In particular $K_U$ is ample and spanned.

The morphism $\pi:U\to \p^2$ given by $|K_U|$ and the inclusion $i:U\to \p^4$
 define an embedding $\alpha=\pi\times i:U\to \p^2\times \p^4$ such that $\alpha^*(\mathcal O(1,1))=K_U+H$. Since $h^0(K_U+H)=4+g(H)-1=12$, the surface $U$ is realized as a divisor of type $(1,2)$ in a  three dimensional linear section $Y\subset\p^2\times\p^4\subset \p^{14}$, i.e. it is the intersection of two divisors of type $(1,1)$ and of a divisor of type $(1,2)$ in $\p^2\times\p^4$. Clearly $\langle K_U,H\rangle\subseteq \Pic(U)$ and we claim that $\Pic(U)=\langle K_U, H\rangle$ if $U$ is very general in the corresponding Hilbert scheme of surfaces in $\p^4$ of degree 9 and of sectional genus. Indeed,  $U\subset\p^2\times \p^4$ deforms to an intersections of three general divisors of the previous type and we can suppose $Y$ smooth if $U$ is general.  Then $\Pic(U)\simeq \Pic (Y)$  by \cite[Theorem 1]{RSr} for a very general general type Castelnuovo surface $U\subset\p^4$ of degree 9 and sectional genus 9, proving the claim because $\Pic(Y)\simeq \Pic(\p^2\times\p^4)$ by Lefschetz Theorem. In conclusion,  for a very general $[X]\in \NL(\mathcal P)$ we have $H^{2,2}(X,\mathbb Z)\simeq\langle h^2, P\rangle$, as expected. The discriminant of $\NL(\mathcal P)$ is defined and equal to 31.
\end{ex}
\medskip

 We now produce a family of surfaces $S\subset\p^7$, contained in  smooth complete intersections of three quadrics  admitting an eight dimensional family  of 5-secant twisted cubics such that the locus of this family through a general point of $\p^7$ describes a unisecant irreducible Edge 3-fold to a general $X\subset\p^7$ through $S$. 
 \medskip

\begin{ex}\label{esempio2} Let $S\subset\p^7$ be a smooth projective surface of degree 9 and sectional genus 3
given as the scheme theoretic  intersection of a general divisor of type $(2,1)$ and of a general divisor of type $(1,2)$ on a Segre fourfold $\p^1\times \p^3\subset\p^7$. A general divisor of type $(2,1)$ is a smooth quintic rational normal scroll $\Sigma\subset\p^7$ while a general divisor of type $(1,2)$ is a so called Edge 3-fold $T\subset\p^7$ of degree 7 and sectional genus 2. Since $S\subset\p^7$ is a linearly normal surface of degree $9=2g+1+2$
with $q(S)=h^1(\mathcal O_S)=0$, it is arithmetically Cohen--Macaulay and satisfies property $N_2$ of Green. Looking at a general 0-dimensional linear section $\Gamma\subset\p^5$, the usual Castelnuovo argument proves that its ideal is generated by $h^0(\mathcal I_{\Gamma}(2))=h^0(\O_{\p^5}(2))-\deg(\Gamma)=21-9=12$ quadratic forms.
One verifies immediately that a complete intersections of three general quadrics through $S\subset\p^7$ is smooth so that it is rational by Theorem \ref{section}.

Let $\psi_1:S\to \p^1$ be the restriction to $S$ of the first projection of $\p^1\times\p^3$. The surfaces  $S\subset\p^7$  contain the pencil of conics given by the fibers of $\pi_1:S\to \p^1$ so that they have  hyperelliptic hyperplane sections. By a theorem of Castelnuovo  they are   projectively equivalent to the projection from seven point lying on different fibers of $\mathbb F_e\subset\p^{14}$ embedded by the linear system $|2C_0+(4+e)F|$, $e=0,1,2,3$ (see \cite[Theorem 4.1]{Ionescu} for a modern proof and for details about the previous notation). In the Hilbert scheme of smooth projective surfaces of degree 9 and sectional genus 3, there is an irreducible component $\mathcal S$, whose general point is a surface with hyperelliptic hyperplane sections as above. The most general $[S]\in \mathcal S$ (or equivalently the most general complete intersection of divisors of type $(2,1)$ and $(1,2)$ in $\p^1\times\p^3\subset\p^7$) corresponds to $e=0$ and such a general $[S]\in\mathcal S$ can be also described as the blow-up of the plane in eight  points $\{p_1,\ldots,p_8\},$ embedded by the linear system of quintic plane curves passing through $p_j$, $j=1,\ldots, 7$ and having a point of multiplicity at least three at $p_8$. In this plane representation, the pencil of conics coming from $\mathbb F_0\subset\p^{14}$ (or from $\psi_1:S\to \p^1$) is represented by the pencil of lines through $p_8$. Let $C\subset S\subset\p^7$ be such a conic.
The Hilbert scheme  $\mathcal S$ is irreducible of dimension $71=8+63$. Since $h^0(N_{S/\p^7})=71$ and $h^1(N_{S/\p^7})=0$ for a general $[S]\in\mathcal S$,  $\mathcal S$ is generically smooth. 

The projection from $C$ defines an embedding  $\psi:S\to \p^4$, represented on the plane by the linear system of quartics passing through $p_1,\ldots, p_7$ and having a double point at $p_8$. Hence $S'=\phi(S)\subset\p^4$ is a smooth surface of degree 5 and sectional genus 2 with hyperelliptic hyperplane sections. We can choose a point $p\in\p^4\setminus S'$ and define a birational morphism $\psi_p:S'\to S''\subset\p^3$ with $S''\subset\p^3$ a singular surface of degree 5. Let $\psi_{2,p}=\psi_p\circ\psi:S\to\p^3$.
The morphism $\psi_1\times\psi_{2,p}:S\to \p^1\times\p^3\subset\p^7$ is given by a sublinear system of $|H|$ of dimension 7, showing that $S\subset\p^7$ is contained in a four dimensional family of Segre 4-folds $\p^1\times \p^3\subset\p^7$ parametrized by (general) points $p\in \p^4\setminus S'$.
The projection of $\Sigma$ from the plane of a conic $\langle C\rangle$ is the unique quadric $Q\subset\p^4$ containing $S'$. If the point $p\in \p^4\setminus Q$, we deduce that we can choose a four dimensional family of Segre four-folds parametrized by general points of $\p^4$ and on which the surface $S$ is a complete intersection of a fixed divisor $\Sigma\subset\p^1\times \p^3$ of type $(2,1)$ (and uniquely determined by the planes of the pencil of conics on $S$) and of a variable divisor $T\subset\p^1\times\p^3$ of type $(1,2)$. The ideal of $S'\subset\p^4$ is generated by the equation of $Q$ and by two cubic forms. It is well known that the cubics through $S'\subset\p^4$ define a birational map $\beta:\p^4\map W\subset\p^6$ with $W$ a smooth complete intersection of two quadrics. So $W$ is explicitly birational to the parameter space of Segre fourfols $\p^1\times\p^3\subset\p^7$ containing $S$ via the previous map and hence  also to the parameter space $\mathcal F$ of Edge threefolds $T$ such that $S=T\cap \Sigma$ on the corresponding Segre fourfold. Let $\pi:\mathcal T\to \mathcal F$ be the universal family over the four dimensional family $\mathcal F$ of Edge threefolds containing $S$ just constructed and let $p:\mathcal T\to \p^7$ be the tautological morphism.

Since $T\subset\p^1\times \p^3$ is a divisor of type $(1,2)$, the second projection restricted to $T$, $\tilde\pi_2:T\to \p^3$, is birational. Indeed, for $[\mathbf x]\in\p^3$ fixed, we have to solve a homogeneous linear equation $a(\mathbf y)f(\mathbf x)+b(\mathbf y)g(\mathbf x)=0$ on $\p^1=\p^1\times [\mathbf x]$. The solution is a unique point except for those $[\mathbf x]$ for which
$f(\mathbf x)=0=g(\mathbf x)$, when it is the whole line $\p^1\times [\mathbf x]$. Hence the base locus of $\gamma=\tilde\pi_2\inv$ is given by $f(\mathbf x)=0=g(\mathbf x)$, which is the complete intersection of two quadrics and hence a curve $D\subset\p^3$ of degree four and arithmetic  genus 1, which is smooth  for $T$ smooth. The map $\gamma:\p^3\map T\subset\p^7$ is given by the linear system of cubics vanishing on $D$. If $L\subset\p^3$ is a general  line, then $\gamma(L)\subset T\subset\p^7$ is a twisted cubic. Since $S''\subset\p^5$ has degree 5, then $5=\length(L\cap S'')=\length(\gamma(L)\cap S)$, that is $T\subset\p^7$ is a locus of 5-secant twisted cubics to $S$. A general quadric $Q_i\subset \p^7$ through $S$, $i=1,2,3$, cuts $T$ in a surface  $S_i\subset T$ of degree $2\cdot \deg(T)-\deg(S)=14-9=5$, which is the restriction to $T$ of a divisor of type $(0,1)=(2,2)-(2,1)$ because $S$ is the intersection of a divisor of type $(2,1)$ with $T$. Hence $S_i=\gamma(P_i)$ with $P_i\subset \p^3$ a plane so that $S_i\subset T\subset \p^7$ is a smooth quintic del Pezzo surface, which is a locus of 5-secant twisted cubics to $S$ images of the lines in $P_i$. By the previous description, the linear system $|H^0(\mathcal I_{S/\p^7}(2))|$ restricted to $T$ gives
the linear system $S+|S_i|$ and hence the birational map $\tilde\pi_2$. Then $Q_1\cap Q_2\cap T=\gamma(P_1\cap P_2)\cup S$ and $Q_1\cap Q_2\cap Q_3\cap T=\gamma(P_1\cap P_2\cap P_3)\cup S$, that is a general $[T]\in\mathcal F$ is unisecant to $X$. We can conclude that a general $X=Q_1\cap Q_2\cap Q_3\subset\p^7$ through $S$  is a section of $\pi:\mathcal T\to \mathcal F
W$ with $\mathcal W$ a rational fourfold (explicitly birational to $\p^4$ and to $W\subset\p^6$).

From the previous analysis one deduces that the four dimensional family $\mathcal W$ of Edge threefolds $T\subset\p^7$ containing $S$ is such that through the general point of $\p^7$ there passes a unique member of the family so that $\mathcal W$ is a four dimensional congruence of Edge threefolds unisecant to $X\subset\p^7$, general complete intersection of three quadrics through $S$.

Let $$\NL(\mathcal S)=\overline{\{[X]\in\mathcal M_{(2,2,2);7}\,:\, S\subset X\text{ with } [S]\in\mathcal S\}}.$$
The Hilbert scheme $\mathcal S$ is irreducible,  has dimension 71 and it is generically smooth. Since $h^0(\mathcal I_{S/\p^7}(2))=12$, the family of complete intersections of three quadrics containing a fixed general $S\subset\p^7$ is isomorphic to $G(3,H^0(\mathcal I_{S/\p^7}(2)))$, which is irreducible of  dimension 27. Since in an explicit  example we verified that $h^0(N_{S/X})=2$, we can conclude that the family of complete intersections containing some $[S]\in\mathcal S$ is irreducible, generically smooth and of dimension at least $71+27-2=96$ and hence of codimension at most three in $\mathcal H_{(2,2,2);7}$ (see Proposition~\ref{conti_parametri}). Moreover $\NL(\mathcal S)$  consists of rational smooth complete intersections of three quadrics because  the general one is rational by Theorem \ref{section} and so we can apply \cite[Theorem 1]{KontsevichTschinkelInventiones}.
\end{ex}

\begin{rmk}{\rm The degree 9 and genus 3 surface $S\subset\p^7$ described in Example \ref{esempio2} admits a  eight dimensional family  $\mathcal D$ of $5$-secant twisted cubics; $F_p\subset\p^7$, the locus of curve in $\mathcal D$ through $p$ is an Edge variety of degree 7 and genus 2 and $\mathcal W$ is explicitly birational to $\p^4$ (and to $W\subset\p^6$ a smooth complete intersection of two quadrics); the locus $S_p\subset Q\subset\p^7$ of curves in $\mathcal D$ through $p$ and contained in a general quadric $Q$ through $S$ is a smooth quintic del Pezzo surface; there exists a unique $[C_p]\in\mathcal D$ passing through a general point $p\in V$ and contained in $V$, where $V\subset\p^7$ is a general complete intersection of two quadrics through $S$.

The definition implies that the restriction of $\alpha:\p^7\map\p(H^0(\mathcal I_{S,\p^7}(2)))$ to $F_p$ induces  a birational map $\alpha:F_p\map\p^3_p\subset \p(H^0(\mathcal I_{S,\p^7}(2)))$ defined by $|S_q|$, $q\in F_p$. Indeed, by construction $Q_i\cap F_p=S_i\cup (F_p\cap S)$ for $i=1,2,3$, with $Q_i\subset\p^7$ general quadric through $S$ and $p$. Then $Q_1\cap Q_2\cap F_p=C_p\cup (F_p\cap S)$ is the unique curve in $\mathcal C$ through $p$ and finally  $Q_1\cap Q_2\cap Q_3\cap F_p=p\cup (F_p\cap S)$, proving the birationality of $\alpha_{|F_p}$. Moreover $\alpha(F_p)$ is a linear space of dimension 3 because we proved that a general linear space of codimension three in $\p(H^0(\mathcal I_{S,\p^7}(2)))$ intersects $\alpha(F_p)$ in a point. 
}
\end{rmk}

The previous definitions are also a  generalization of those introduced in \cite{RS1} to prove the rationality of cubics fourfolds belonging to some Hassett divisors, see also \cite{JEMS}. We now adapt  to this setting also the tools introduced in \cite[Section 1]{RS1} and in \cite{Explicit, JEMS}. 

Let $X\subset\p^7$ be a smooth rational complete intersections of three quadrics and let notation be as in \S \ref{birass} .
Since $p:\mathcal V\to \p^7$ is birational, we can define the rational map $$\varphi=\pi\circ p^{-1}:\p^7\map\mathcal W,$$ whose  fiber through a general $q\in \p^7$, $F_q=\overline{\varphi^{-1}(\varphi(q))}=p(\pi\inv(\phi(q)))$, is the unique irreducible 3-fold of the family $\mathcal W$ passing through $q$. 

Following \cite{Explicit,JEMS}, it is natural to look for the  linear systems on $\p^7$ giving the abstract rational maps $\varphi: \p^7\map \mathcal W$.
The
 linear systems  $|H^0(\mathcal I_{S}^e(2e-1))|$, $e\geq 1$, when not empty, define a rational map
 \begin{equation}\label{defmu}
 \mu:\p^7\map W\subseteq \p(H^0(\mathcal I_{S}^e(2e-1))),
 \end{equation} 
  which  contracts the curves in $\mathcal C$ and hence contracts to a point also the loci $F_p$, $p\in\p^7$ general,  because, by definition, through $q_1,q_2\in F_p$ general points there passes a curve $C_{q_1,q_2}\subset F_p$ with  $[C_{q_1, q_2}]\in\mathcal C$. If $\dim(W)=4$ (as it happens for all the known congruences), then $\mu$ provides a birational geometric realization of $\varphi$ (the irreducible component of $\overline{\mu\inv(\mu(p))}$ through a general $p\in\p^7$ is $F_p$) and produces a birational model $W$ of $\mathcal W$. Quite surprisingly,  in all the examples considered until now  we always found that the image  $W\subseteq\p^N$ is smooth and also a well known rational prime Fano fourfold.  
  
  Since the picture in Example \ref{esempioPiano} is clear from this point of view ($e=1$ and the linear system  $|H^0(\mathcal I_P(1))|$ gives the projection from $P$), let us see how  this works concretely  in Example \ref{esempio2}.

\begin{ex}({\bf Example \ref{esempio2} continued})\label{esempio2cont} Let notation be as in Example \ref{esempio2}. We constructed a congruence $\mathcal V\to \mathcal W$ of unisecant Edge threefolds to a general complete intersection of three quadrics through a degree 9 and genus 3 surface $S\subset\p^7$ with $\mathcal W$ birational to $\p^4$ and a diagram of rational maps
\begin{equation}\label{diagramma congruenza FpW}
\xymatrix{
 \mathcal V\ar[r]^{p}\ar[d]_{\pi}&\p^7\ar@{-->}[ld]_{\phi}\\
\mathcal W\ar@{-->}[r]_{\beta}&W,}
\end{equation}
with $W\subset\p^6$ a smooth complete intersection of two quadrics and with $\beta:\p^4\map W$ the maps given by cubics through $S'\subset\p^4$, the projection of $S$ from a conic $C\subset S\subset\p^7$. An explicit computation shows that for a general $[S]\in \mathcal S$ we have $|H^0(\mathcal I_S^3(5))|=\p^6$, that the image of $\mu:\p^7\map\p^6$ is a smooth complete intersection of two quadrics and that, for $p\in\p^7$ general, $\overline{\mu\inv(\mu(p))}=F_p$. In conclusion, the rational map $\mu$ is such that the following diagram of rational maps  commutes:
\begin{equation}\label{diagramma congruenza Fp2}
\xymatrix{
 \mathcal H\ar[r]^{p}\ar[d]_{\pi}&\p^7\ar@{-->}[d]^{\mu}\ar@{-->}[ld]_{\phi}\\
\mathcal W\ar@{-->}[r]_{\beta}&W.}
\end{equation}
Hence, for a general $[X]\in G(3,H^0(\mathcal I_{S/\p^7}(2)))$ the birational map $\mu_{|X}:X\map W\subset\p^6$  gives an explicit realization of the abstract birational map $\varphi_{|X}:X\map \mathcal W$ composed with the birational map $\beta:\mathcal W\map W$.
\end{ex} 
After  constructing the birational maps $\mu_{|X}: X\map W$, which realize geometrically the maps $\pi\circ p\inv:X\map \mathcal W$, one may ask about their factorization via the most elementary links according to the Sarkisov Program.

\section{The secant flop and the extremal congruence contraction} \label{s2}

We first introduce and study the behaviour of secant lines to a non degenerate irreducible projective surface $S\subset\p^7$ and contained in a smooth intersection of three quadrics $X\subset\p^7$. Then necessarily the secant variety $\Sec(S)$ has dimension 5. To contain a secant line to $S$ imposes only one condition to  quadric hypersurfaces in $\p^7$ through $S$ so that
$X\subset\p^7$  is expected to contain a one dimensional family of secant lines to $S$ describing a surface $T\subset X$. Since $T=\Sec(S)\cap X$  and since
$\Sec(S)\cap X$ is irreducible and of dimension two for $X\subset\p^7$  general (inside $G(3,H^0(\mathcal I_S(2)))$),  we expect that  $T$ is irreducible and of dimension two for $X$ general.

\subsection{Assumptions}\label{secbeh}   Suppose $S\subset\p^7$ is  a smooth irreducible projective surface  with $\dim(\Sec(S))=5$ and scheme-theoretically defined by quadric hypersurfaces
such that: the associated rational map
$$\phi:\p^7\map\p(H^0(\I_{S/\p^7}(2))=\p^N$$
is birational onto the closure of its image $Z=\overline{\phi(\p^7)}\subset\p^N$; $\phi$ defines an isomorphism outside $\Sec(S)$; $\phi$ contracts $\Sec(S)$ to a variety of dimension four.

Let $$\tilde\phi:X'=\Bl_SX\to Y\subset\p^{N-3}$$ be the induced birational morphism with  $X\subset\p^7$ a general complete intersection of three quadrics through $S$ and with $Y$ the corresponding linear section of $Z$ . Under the previous  hypothesis $\tilde\phi$ is an isomorphism outside
the strict transform $T'$ of $T=\Sec(S)\cap X$. We shall suppose that $T$ (and hence $T'$) is irreducible of dimension two so that $\tilde C=\tilde\phi(T')$ is an irreducible curve and  $\tilde\phi:T'\to \tilde C$ is generically a $\p^1$-bundle. The last assumption  is  very natural and mild, as explained above.

If $\lambda:X'=\Bl_SX\to X$, then
$$-K_{X'}=\lambda^*(-K_X)-E=2H'-E$$
with $H'=\lambda^*(H)$ and with   $H\subset X$  a hyperplane section.  Our hypothesis on the defining equations of $S$ and on the birational map $\phi:\p^7\map Z$ imply that
 $-K_{X'}$ is a big divisor generated by its global sections so that $X'$ is a log-Fano manifold with $\rho(X')=2$ and in particular
 a Mori Dream Space. Moreover,
 the induced morphism 
$$\tilde\phi:\Bl_SX\to Y$$
 is  a small contraction with irreducible exceptional locus $T'$ defined by the extremal ray $[L']$, where 
$L'\subset X'$ is the strict transform of a proper secant line to $S$ contained in $X$. 

The $\p^1$-bundle $\tilde\phi:T'\to \tilde C$ inside $X'=\Bl_SX$ 
can be flopped to produce the small contraction $\tilde\psi:W'\to Y$ with $W'$ a smooth projective fourfold not isomorphic to $X'$
by general results on Mori Dream Spaces with $\rho=2$ (or by a direct analysis as in \cite{JEMS}). 
These general facts motivate the following definitions.

\begin{defin}{\rm ({\bf Secant flop}) Let notation and assumptions be as above. The small contraction $\tilde\phi:X'\to Y$ of  curves in the extremal ray $[L']$
is called a {\it secant  flop contraction}. The associated  flop contraction $\tilde\psi:W'\to Y$ of $\tilde\phi$ with $W'$ a projective fourfold not isomorphic to $X'$
defines a birational
map $\tau=\tilde\psi\inv\circ\tilde\phi:X'\map W'$  called the {\it secant flop} of $\tilde\phi:X'\to Y$.} 
\end{defin}

Let $\lambda: X'=\Bl_SX\to X$ be the blow-up
of $X$ along $S$, let $E\subset X'$ be the exceptional divisor and let $H'=\lambda^*(H)$ be as above.

 So far we constructed  a diagram 
\begin{equation}\label{diagramma2}
\UseTips
 \newdir{ >}{!/-5pt/\dir{>}}
 \xymatrix{
 &\Bl_{T'}X'=\Bl_{R'}W'\ar[ld]_\sigma\ar[rd]^{\omega}&\\            
 X' \ar@{-->}[rr]^{\tau}\ar[rd]_{\tilde\phi}\ar[d]_\lambda&                             &W'\ar[ld]^{\tilde\psi}\\
X &Y&
}
\end{equation}
with $R'\subset W'$ a smooth ruled surface such that $\tilde\phi(T')=\tilde C=\tilde\psi(R')$. 

\subsection{Secant flop and congruences of \texorpdfstring{$(2e-1)$}{(2e-1)}-secant rational curves of degree $\mathbf e$} 
The aim of this section is to relate the secant flop to (congruences of) $(2e-1)$--secant curves to $S\subset X$ determined by a congruence of unisecant 3-folds to $X$. Reasoning as in  \cite[Proposition 2.9]{JEMS}\label{extremalray}, one proves that: if $C'\subset X'$ is the strict transform on $X'$ of a $(2e-1)$-secant curve $C$ to $S$ of degree $e$ contained in $X$, then the strict transform $\overline C'$ of $C'$ on $W'$ generates an extremal ray on $W'$, giving the following result.

\begin{thm} {\rm ({\bf Extremal contraction of the congruence}, see also \cite[Theorem 2.10]{JEMS})}\label{contraction}  Let notation  and assumptions on $S\subset\p^7$ be as above. Suppose a general $X\subset\p^7$ through $S$ admits  a congruence of unisecant 3-folds producing an eight dimensional family $\mathcal D$ of curves of degree $e$ which are $(2e-1)$-secant to $S$. If $D\subset X\subset\p^7$ is the irreducible divisor given by the locus  of  curves in $\mathcal D$ contained in $X$, then:

\begin{enumerate}

\item  there exists a divisorial contraction $\nu:W'\to W$, with
$W$ a locally $\mathbb Q$--factorial projective Fano  variety, whose exceptional locus $\overline E$ is the strict transform of $D$ on $W'$ and such that $\nu(D)=U$ is an irreducible surface supporting  the base locus scheme $B$ of $\nu\inv$.  The base locus scheme $B$ is generically smooth, irreducible and $\nu$ is generically the blow--up of the surface $U$. 

\item If $\mu=\nu\circ\tau\circ\lambda\inv$, then the  irreducible surface $T\subset X$ is contained in the base locus scheme of $\mu$ and its flopped image $R=\nu(R')\subset W$ is is a surface  ruled by  {\it lines} in $W$, which are $i(W)$-secant to $U$, and $R$ is contained in the base locus scheme of $\mu\inv$.
\end{enumerate}
Therefore diagram  \eqref{diagramma2} is completed into:

\begin{equation}\label{diagramma3}
 \UseTips
 \newdir{ >}{!/-5pt/\dir{>}}
 \xymatrix{
 &\Bl_{T'}X'=\Bl_{R'}W'\ar[ld]_\sigma\ar[rd]^{\omega}&\\             
 X' \ar@{-->}[rr]^{\tau}\ar[rd]^{\tilde\phi}\ar[d]_\lambda &                             &W'\ar[ld]_{\tilde\psi}\ar[d]^\nu\\
X\ar@{-->}[r]^{\phi}\ar@{-->}@/_0.6cm/[rr]_{\mu} &Y& W \ar@{-->}[l]_{\psi}.   }
\end{equation}

If  $W$ and $U\subset W$ are smooth, the previous diagram induces  isomorphisms of Hodge structures between
\begin{equation}\label{isoH4}
H^j(X,\mathbb Z)\oplus_{\perp}H^{j-2}(S,\mathbb Z)(-1)\oplus_{\perp}H^{j-2}(T',\mathbb Z)(-1)
\end{equation}
and
\begin{equation}\label{isoH4W} H^j(W,\mathbb Z)\oplus_{\perp}H^{j-2}(U,\mathbb Z)(-1)\oplus_{\perp}H^{j-2}(R',\mathbb Z)(-1),
\end{equation}
for every $j=0,\ldots, 8$,
which for $j=4$ restricts to the isomorphism:
\begin{equation}\label{isoT}
\mathcal T_{X}\oplus_{\perp} \mathcal T_S(-1)\simeq \mathcal T_{W}\oplus_{\perp}\mathcal T_U(-1).
\end{equation}
\end{thm}

\begin{rmk}\label{Mukai}{\rm When  the secant flop exists (which is the case under mild hypothesis), the existence of the family   of $(2e-1)$-secant lines to $S$ (that is the existence of the congruence of unisecant 3-folds to a general $X\subset\p^7$ through $X$)  is equivalent to the existence  of the surface $U\subset W$, which a posteriori appears as a different (birational) incarnation of $S_X$, the associated Castelnuovo surface of general type to $X$ (which is minimal). So one can expect that (at least) $\mathcal T_X\simeq\mathcal T_U(-1)$ holds. From this perspective one associates to the pair $(X,S)$ a pair $(W,U)$, where $W$ is the image of $X$ via $\mu$ and where the surface $U$ is birational to  the parameter space of the curves of the congruence $\mathcal C$ contained in $X$ and describing the divisor $D$. Moreover, in many cases treated here, the induced isomorphism in \eqref{isoT} sends $H^4(X,\mathbb Z)$ into  $H^2(U,\mathbb Z)(-1)$ and $H^4(W,\mathbb Z)$ into $H^2(S,\mathbb Z)(-1)$, Examples \ref{esempioPiano}, \ref{esempio2}). In  Theorem  \ref{esempioC14} we shall see that   $H^4(X,\mathbb Z)\simeq H^2(U,\mathbb Z)(-1)$ and $H^4(W,\mathbb Z)\simeq H^2(S,\mathbb Z)(-1)$ with $S\subset X$ a non minimal K3 surface and with $W$ a smooth cubic fourfold so that in this case all the terms in \eqref{isoT} appear.
Example \ref{CI6} shows that  image of $H^4(X,\mathbb Z)$ may be strictly contained in $H^2(U,\mathbb Z)(-1)$ (this happens as soon as  $U$ is non minimal).
}
\end{rmk}

\begin{ex}({\bf Example \ref{esempio2cont} continued})\label{CI6} Let notation be as in Example \ref{esempio2}. Let $S\subset\p^7$
be the degree 9 and genus 3 surface admitting a congruence of 5-secant twisted cubics constructed  in Example \ref{esempio2cont}.
For a general $[S]\in \mathcal S$ we have $|H^0(\mathcal I_S^3(5))|=\p^6$ and that the image of $\mu:\p^7\map\p^6$ is a smooth complete intersection of two quadrics $W\subset\p^6$. The surface $S\subset\p^7$ has the expected secant behaviour so that, for a general $[X]\in G(3,H^0(\mathcal I_{S/\p^7}(2)))$, the birational map $\mu=\mu_{|X}:X\map W\subset\p^6$  is a secant flop by Theorem \ref{contraction}. An explicit computation shows that $U\subset W\subset\p^6$ is a smooth surface of degree 16 and sectional genus 14, whose ideal is generated by two quadrics and 9 cubics. Let us recall that in this case $\mathcal T_W=0$ and $H^4(W,\mathbb Z)\simeq \mathbb Z^{8}$ is generated by the classes of the $64$ distinct planes contained in $W$, see \cite{Reid}. To see directly that $H^4(W,\mathbb Z)\simeq \mathbb Z^8$ and $\mathcal T_W=0$, one can also use the
birational parametrization $\beta:\p^4\map W\subset \p^6$ given by $|H^0(\mathcal I_{S'/\p^4}(3))|=\p^6$ and the fact that $\beta\inv:W\map \p^4$ is the projection
from a line $L\subset W$, yielding $\Bl_{S'}\p^4\simeq \Bl_LW$ and hence $\mathcal T_{W}\simeq \mathcal T_{\p^4}\oplus \mathcal T_{S'}=0$. Moreover,  $S'\simeq S$ and  
$$\displaystyle
H^2(S,\mathbb Z)\simeq H^2(\p^2,\mathbb Z)\oplus_\perp(\oplus_{i=1}^8\mathbb Z[E_i])
$$ since  $S\simeq \Bl_{p_1,\ldots,p_8}\p^2$. Furthermore, $h^{3,1}(W)=h^{3,1}(\Bl_LW)=h^{3,1}(\Bl_{S'}\p^4)=h^{3,1}(\p^4)+h^{2,0}(S')=0$ and 
$h^{2,1}(W)=h^{2,1}(\p^4)+h^{1,0}(S')=0$.
From \eqref{isoH4} we deduce the isomorphism of Hodge structures between
$$
H^4(X,\mathbb Z)\oplus_{\perp}H^2(S,\mathbb Z)(-1)\oplus_{\perp}H^{2}(T',\mathbb Z)(-1)$$
and $$  H^4(W,\mathbb Z)\oplus_{\perp}H^2(U,\mathbb Z)(-1)\oplus_{\perp}H^{2}(R',\mathbb Z)(-1).
$$
In particular, $$p_g(U)=h^{2,0}(U)=h^{3,1}(\Bl_UW)+h^{3,1}(\Bl_SX)=h^{3,1}(X)+h^{2,0}(S)=3$$ and 
$$q(U)=h^{1,0}(U)=h^{2,1}(\Bl_UW)=h^{2,1}(\Bl_{R'}(W')-h^{1,0}(R')=h^{3,1}(\Bl_{T'}X')-h^{1,0}(T')=$$
$$=h^{2,1}(\Bl_SX)=h^{2,1}(X)+h^{1,0}(S)=0.$$
Moreover,  $b_2(U)=b_2(X)+b_2(S)-b_2(W)=44+9-8=45$ and, since $q(U)=0$,  $b_1(U)=b_3(U)=0$, yielding  $\chi_{\top}(U)=2+45=47$. 
The isomorphism of Hodge structures between  \eqref{isoH4} and \eqref{isoH4W} sends $H^4(X,\mathbb Z)$ into $H^2(U,\mathbb Z)(-1)$ and induces $\mathcal T_X\simeq T_U(-1)$, yielding  $H^2(U,\mathbb Z)(-1)\simeq H^4(X,\mathbb Z)\oplus_{\perp}\mathbb Z[D](-1)$
with $D\subset U$ algebraic cycle and $H^{1,1}(U)\simeq H^{2,2}(X)\oplus_{\perp}\mathbb Z[D](-1)$.  
Since $\chi(\mathcal O_U)=4$, Noether Formula yields $K_U^2=12\chi(\mathcal O_U)-\chi_{\top}(U)=1$ which together with $p_g(U)=3$ imply that $U$ is a  surface of general type because it is irrational and with $K^2>0$. Moreover, $U$ is not minimal because  $1=K_U^2<2p_g(U)-4=2$, see \cite[Exercise X.13 (1)]{Beau}. Let $\pi:U\to \widetilde U$ be the contraction of a (-1)-curve $E\subset U$. Since $H^4(X,\mathbb R)\simeq H^2(S_X,\mathbb R)$ and since $S_X$ is a minimal surface, the class $[E]$ does not belong to $H^4(X,\mathbb R)\subsetneq H^2(U,\mathbb R)(-1)$. Then $H^2(\widetilde U,\mathbb R)\simeq H^2(S_X, \mathbb R)$, yielding that $\widetilde U$ is a minimal surface of general type with 
$K_{\widetilde U}^2=2$, $p_g(\tilde U)=3$, $q(\tilde U)=0$ and $\chi_{\top}(\tilde U)=46$, i.e. $\widetilde U$ is a minimal Castelnuovo surface of general type with $K^2=2$.

The fiber over a general point of
$E$ of the map $\mu:\p^7\map W\subset\p^6$ has dimension four and intersects a general $X\subset\p^7$ through $S$  into the curve  $E$, which  is then contained in the corresponding $U$, which one of the two irreducible components of $\mu_{|X}\inv$. Hence 
the curve $E\subset U$ is contained also in a general $U'\subset W$ determined by a general $X'$ through $S$. So $E\subseteq U\cap U'$ and we have an explicit method to determine it by cutting with different $U$'s. By choosing two general $X, X'$ as above, one computes that $E\subset U\subset \p^6$ is a line. Hence $U\subset \p^6$ is the projection from an internal point of  a smooth minimal Castelnuovo surface $\widetilde  U\subset \p^7$ of degree 17 and sectional genus 14, whose ideal is generated by seven quadratic equations.

\end{ex}

We shall consider in detail an example, already appeared in \cite[Example viii, Table 2]{JEMS} to produce a Trisecant Flop and analyzed here from a different perspective.

\begin{Proposition}\label{NLK}
Let $S\subset\p^7$ be the projection of a general minimal K3 surface $\tilde S\subset\p^8$ 
of degree 14 and genus 8 from a general point on $\tilde S$, let $\mathcal S$ be the irreducible component of the Hilbert scheme 
to which a general $S\subset\p^7$ as above belongs and let  
$$\NL(\mathcal S)=\overline{\{[X]\in\mathcal H_{(2,2,2);7}\,:\, S\subset X\text{ with } [S]\in\mathcal S\}}\subset \mathcal H_{(2,2,2);7}.$$

Then $\NL(\mathcal S)$ is irreducible, of codimension at most 3  and contained in $\Rat(\mathcal H_{(2,2,2);7})$.
\end{Proposition}

\begin{proof} The smooth surface $S\subset\p^7$ has homogenous ideal generated by 9 quadratic forms and it is contained in a smooth complete intersection of three quadrics. 
Moreover,  $h^1(N_{S/\p^7})=0$,  $h^0(N_{S/\p^7})=84$, $h^1(\mathcal O_S(2))=0$ and in an explicit  example we verified that $h^0(N_{S/X})=6$ so that the estimate about the codimension follows from Proposition \ref{conti_parametri}. Theorem \ref{section} implies  that a general $[X]\in \NL(\mathcal S)$ is rational so that  \cite[Theorem 1]{KontsevichTschinkelInventiones} assures that $\NL(\mathcal S)$ is contained in  $\Rat(M_{(2,2,2),7})$. 
\end{proof}

 We provide a direct proof of the rationality of a general $[X]\in\NL(\mathcal S)$ and a description of the geometry in \eqref{diagramma3} and in \eqref{isoT} via the tools introduced in this section (and, in some point, with the aid of some M2 computations). 
 
 \begin{Proposition}\label{mu0} A  surface $[S]\in\mathcal S$ admits a eight dimensional family $\mathcal D$ of rational curves of degree 5 which are 9-secant to $S$ and such
 that the locus of curves in $\mathcal D$ passing through a  general point $p\in\p^7$ is an irreducible 3-fold $F_p\subset\p^7$ of degree 11, sectional genus 7 and with 4 singular points.
 
 The irreducible 3-folds $F_p$'s are unisecant to a general $[X]\in\NL(\mathcal S)$ through $S$ and form a congruence whose parameter space is birational to a smooth cubic fourfold $W\subset \p^5$, which is the closure of the image of the map $\mu:\p^7\map \p(H^0(\mathcal I_S^5(9)))=\p^5$. In particular, a general $X\subset\p^7$ through $S$ is birational to $W$ via $\mu_{|X}$.
 \end{Proposition}
 \begin{proof}
The surface $S\subset\p^7$ is a general hyperplane section of a smooth Fano 3-fold $D\subset\p^8$, whose ideal is generated by 8 quadratic equations defining a special Cremona transformation $\alpha:\p^8\map\p^8$ of type $(2,5)$, meaning that $\alpha$ is solved by blowing-up  $\tilde D$ and that the inverse is defined by forms of degree 5 (see \cite[Theorem 4.4]{HKS}). Hence the map $\phi:\p^7\map Z\subset\p^8$ given by $|H^0(\mathcal I_{S/\p^7}(2))|$, being the restriction of $\alpha$ to a general hyperplane in $\p^8$, is birational onto its image $Z\subset\p^8$, which is a quintic hypersurface. 
 Let $z=\phi(p)\in Z$ be a general point and let $B=\phi(\Sec(S))\subset Z$ be the base locus of $\phi^{-1}$, which is irreducible of dimension 4. Any irreducible curve $C\subset\p^7$ of degree $e\geq 1$ which is $(2e-1)$-secant to $K$ and which passes  through $p$ is mapped by $\phi$ onto a line $L=\phi(C)\subset Z$ passing through $z=\phi(p)$. Since $\phi\inv:Z\map\p^7$ can be represented by the restriction of a linear system of quintic hypersurfaces, we deduce that $e\leq 5$.

The Hilbert scheme $\mathcal L_{z,Z}\subset \p^6=\p(t_zZ)$ of lines passing through $z$ and contained in $Z$
is a complete intersection of dimension 2 and degree $5!=120$. Indeed, by considering a general complete intersection $V\subset \p^7$ of two quadrics through $S$ and by restricting $\phi$ to $V$ we get a birational map $\phi:V\map Z'\subset \p^6$ with $Z'\subset\p^6$ a general linear section of $Z\subset\p^8$.  Then for a general point $z'=\phi(v)$ with $v\in V$ general, we verified that $\mathcal L_{z',Z'}\subset \p^4=\p(t_{z'}Z')$ consists of 120 distinct lines $L_i\subset Z'$ so that $\mathcal L_{z',Z'}\subset \p^4=\p(t_{z'}Z')$, being scheme-theoretically defined by four equations,  is a smooth complete intersection of degree $5!=120$ (although $Z'\subset\p^6$ is singular). Hence $\mathcal L_{z,Z}\subset \p(t_zZ)=\p^6$ is a complete intersection of the four equations defining it scheme-theoretically. 

The preimages on $V$ via $\phi$ of the lines in $\mathcal L_{z',Z'}$ provide: thirteen 1-secant lines to $S$ through $v$; thirty-eight 3-secant conics to $S$ through $v$; fifty-four 5-secant twisted cubics to $S$ through $v$; fourteen 7-secant quartic curves to $S$ through $v$; one 9-secant quintic curve to $S$ through $v$. In particular, $\mathcal L_{z,Z}$ contains an irreducible component of dimension two isomorphic to $\p^2$, consisting of lines $L$ through $z$ such that $L\cap B=\emptyset$ and whose locus  in $Z$ is a linear space $\p^3_z\subset Z$ passing through $z$. Hence, there exists an irreducible component of the Hilbert scheme $\mathcal L_Z$ of lines contained in $Z$ having dimension 8 and such that for a general line $[L]\in\mathcal L_Z$ we have that $C=\phi\inv(L)$ is a quintic curve 9-secant to $S$. This produces an irreducible family $\mathcal D$ of dimension 8 of quintic curves 9-secant to $S\subset\p^7$ such that the locus of these curves through a general point $p\in\p^7$ is an irreducible rational 3-fold $F_p=\phi\inv(\p^3_{\phi(p)})$ defined by five quadratic equations in $H^0(\mathcal I_{S/\p^7}(2))$ and unisecant to a general complete intersection of three quadrics through $S$. To prove that the family $\mathcal W$ of irreducible unisecant 3-folds is a congruence we consider the map $\mu:\p^7\map\p^5=\p(H^0(\mathcal I_S^5(9)))$. An explicit computation shows that $W=\mu(\p^7)\subset\p^5$ is a smooth cubic fourfold  and that $\overline{\mu\inv(\mu(p))}$ is an irreducible 3-fold of degree 11, sectional genus 7, with 4 singular points and defined by 5 quadratic equations vanishing on $S$ so that $F_p=\overline{\mu\inv(\mu(p))}$. The family $\mathcal W$ of these 3-folds is  birationally parametrized by $W$ and $F_p\subset\p^7$ is  the locus of curves in $\mathcal D$ passing through $p$. This also proves that 
a general complete intersection of three quadrics $X\subset\p^7$ through $S$ is birational to $W$ via the restriction of $\mu$ to $X$. 
\end{proof}

We recall the definition of the Fano divisor $\mathcal C_{14}$ of {\it cubic fourfolds of discriminant 14} inside the moduli space $\mathcal C$ of smooth cubic fourfolds in $\p^5$.
The Hodge Theoretic definition given by Hassett is the following:
$$\mathcal C_{14}=\{[W]\in\mathcal C\;:\; \exists\; \Lambda \subseteq H^{2,2}(W), \; h^2\in\Lambda,\;\rk(\Lambda)=2,\;  \disc(\Lambda)=14\},$$
see \cite{Levico}. The fact that $\mathcal C_{14}$ is an irreducible divisor is not so immediate, see \cite{Levico} for details. Moreover, for a very general
$[W]\in\mathcal C_{14}$ one has $H^{2,2}(W)=\Lambda$, see {\it loc. cit.}

There are also other geometric descriptions of $\mathcal C_{14}$ as
$$\mathcal C_{14}=\overline{\{[W]\in\mathcal C\;:\; \exists\; U\subset  W\;:\; [U]\not \in \langle h^2\rangle, \; \disc(\langle h^2,U\rangle)=3\cdot S^2-(\deg(U))^2=14\}}$$
with $U\subset\p^5$ suitable irreducible surfaces with at most a finite number of nodes. For example Fano showed that one can take as $U$ either
a smooth quartic rational normal scroll or a smooth quintic del Pezzo surface. Using the geometric description one proves that the locus on the right is a divisor in $\mathcal C$
(and hence an irreducible divisor) via a parameter count analogous to the one performed in the proof of Proposition \ref{NLK}. 

A polarized K3 surface $(\tilde S, \tilde H)$ is {\it associated} to the pair $(W,\Lambda)$ with $W\subset\p^5$ smooth cubic fourfold and with
$\langle h^2\rangle\subseteq\Lambda\subset H^{2,2}(W)$ with $\rk(\Lambda)=2$ if there exists an isomorphism 
$$H^2(\tilde S,\mathbb{Z})(-1)\supset \tilde H^{\perp} 
\xrightarrow{\simeq} \Lambda^{\perp}\subset H^4(W,\mathbb Z)$$
respecting Hodge structures. In this case necessarily $\disc(\Lambda)=(\tilde H)^2.$ 
\medskip

\begin{thm} \label{esempioC14}  A general complete intersection of three quadrics $X\subset\p^7$ containing a general internal projection $S\subset\mathbb{P}^7$ 
  of a general K3 surface $\widetilde{S}\subset\mathbb{P}^8$ of degree $14$ is birational to a cubic fourfold $W\subset\mathbb{P}^5$ containing a smooth minimal Castelnuovo surface $U$ of degree $13$.
Moreover, $[X]\in \mathcal M_{(2,2,2);7}^{[13,28]}$, $[W]\in \mathcal{C}_{14}$, 
$U\subset W$ is the associated Castelnuovo surface to $X$ and $\widetilde{S}$ is the associated K3 surface to $W$. 
\end{thm}  
\begin{proof} A general surface  $S\subset\p^7$ has the expected secant behaviour.
Indeed, one can verify via the  map $\phi:\p^7\map Z\subset\p(H^0(\mathcal I_S(2)))=\p^8$ that the restriction to a general complete intersection $X\subset\p^7$ of three quadrics through $S$
has irreducible exceptional locus equal to $\Sec(S)\cap X$.
 By Theorem \ref{contraction} the birational map $\mu=\mu_{|X}:X\map W$ is a Secant Flop, whose base locus consists of $S$ and of  the irreducible ruled surface $T=\Sec(S)\cap X\subset X$. Let notation be as in \eqref{diagramma3}. Then a computation shows  that $U\subset W\subset\p^5$ is a smooth irreducible  surface  of degree 13, sectional genus 12 and whose ideal is generated by 7 cubic forms. 

From the isomorphisms between \eqref{isoH4} and \eqref{isoH4W}, recalling that $h^{3,1}(W)=1=h^{2,0}(S)$, we deduce
$$p_g(U)=h^{3,1}(\Bl_UW)-h^{3,1}(W)=h^{3,1}(\Bl_S X)-h^{2,0}(S)=h^{3,1}(X)=3.$$ Analogously $q(U)=0$. Since
$b_2(S)=b_2(\tilde S)+1=23$ and since $b_2(W)=23$ the last part of Theorem \ref{contraction} implies
$b_2(U)+b_4(W)=b_2(S)+b_2(X)$, yielding $b_2(U)=b_2(X)=44$ and hence $\chi_{\top}(U)=46$. Then Noether Formula gives $K_U^2=2$ which plugged in formula \eqref{formulona per S2X} together with the previous values
provides $(U)_W^2=61$. The discriminant of the lattice 
\begin{equation}\label{latticeincl}
\langle h^2,\; U\rangle\subseteq H^{2,2}(W)
\end{equation}
 is $3\cdot 61-13^2=14$, proving that $[W]\in\mathcal C_{14}$. The isomorphism in Theorem \ref{contraction} yields $H^4(W,\mathbb Z)\simeq H^2(S,\mathbb Z)(-1)\simeq (H^2(\tilde S, \mathbb Z)\oplus_\perp\mathbb Z[E])(-1)$, where $E\subset K$ is the unique $(-1)$-curve, and $H^2(U,\mathbb Z)(-1)\simeq H^4(X,\mathbb Z)$. Via the previous
 isomorphism $\langle h^2, U\rangle\subseteq H^{2,2}(W)$ is sent into $\langle H, E\rangle\subseteq H^{1,1}(S)$. From these two facts it immediately follows that $\tilde S$ is the {\it associated} K3 surface to $W$ and that $U$ is a minimal surface and hence a minimal Castelnuovo surface of general type with $K^2=2$ because $\chi=1-q(U)+p_g(U)=4$. 
\end{proof}

\begin{rmk}{\rm 
The surface $U\subset\p^5$ appearing in Theorem~\ref{esempioC14} 
 was constructed in a different way in \cite{StaCubicCremona} (see also \cite{Sta19}). Indeed $U$ can be realized as a hyperplane section of a smooth conic bundle in $\p^6$, which is the base locus 
 of one of the four types of special Cremona transformations of $\p^6$.}
\end{rmk}

\begin{rmk}\label{finalrmk}{\rm If $[S]\in \mathcal S$ is very general, then $\Pic(S)\simeq \langle H, E\rangle$ and $H^{2,2}(W)\simeq \langle h^2, U\rangle.$

For a very general $U\subset\p^5$ as above, we have  $\Pic(U)\simeq \langle K_U, H\rangle$ (by the previous remark such a $U$ is a hyperplane section of a 3-fold with Picard group isomorphic to $\mathbb{Z}^2$), yielding $H^{2,2}(X)\simeq \langle h^2, K\rangle$ for a general $X$ through the very general $[S]\in\mathcal S$.

Let notation be as in Theorem \ref{contraction} and let $\mu':X'=\Bl_SX\map W$ be the birational map induced by $\mu_{|X}$ and let $F$ be a fiber of $\lambda_{|E}:E\to S$. Then $F'=\mu'(F)\subset W$ is a rational curve of degree 5 because $(9H'-5E)\cdot F=5$. Letting $\overline F=\tau(F)$ we have, with obvious notation,
$$1=-K_{X'}\cdot F=-K_{W'}\cdot \overline F=(3\overline H-\overline E)\cdot \overline F=15-\overline E\cdot \overline F,$$
yielding that $F'$ is 14-secant to $U\subset \p^5$. Fixing  $U$ and varying the cubic  $W$ through $U$ we conclude that  at least a 14-secant curve of degree 5 to $U$ passes through a general point of $\p^5$.
Since  these curves are contracted to points by the extension of $\mu_{|X}\inv$ to $\p^5$, there is a unique 14-secant curve of degree 5 to $U$ through a general point 
of $\p^5$,
that is $U\subset\p^5$ admits a congruence of 14-secant curves of degree 5. Since the fibers $L'$ of the ruling  of $R'\subset W'$ are  contracted to points by by $-K_{W'}$, we have
$0=-K_{W'}\cdot L'=(3\overline H-\overline E)$, yielding $\overline E\cdot L'=3$. Hence $\nu(R')$ consists of trisecant lines to $U$ contained in $W$.  In conclusion $\mu_{|X}\inv: W\map X$ is a Trisecant Flop according to the terminology of \cite{JEMS} and is given by a linear system of dimension 7 of forms of degree 14 having points of multiplicity 5 along $U$.  
}
\end{rmk}

\section{Summary tables of examples}

\begin{table}[htbp]
  \renewcommand{\arraystretch}{1.0}
  \centering
  \tabcolsep=1.5pt 
  \footnotesize
  \begin{tabular}{|c|c|c|c|}
  \hline
  Surface $S\subset X\subset \mathbb{P}^7$  &
  $\det \left(\begin{smallmatrix} H_X^4& H_X^2\cdot [S]\\ [S] \cdot H_X^2 & (S)_X^2 \end{smallmatrix}\right)$ &
  \begin{tabular}{c} Codim in \\ $\mathcal{M}_{(2,2,2),7}$ \end{tabular} &
  \begin{tabular}{c} Count of parameters \\ ($h^0(\mathcal I_{S/\mathbb{P}^7}(2))$,
       $h^0(N_{S/\mathbb{P}^7})$, 
      \\ $h^0(N_{S/X})$)\end{tabular} \\
  \hline
  \hline
  $\mathfrak{S}(3;5, 0, 0)$ & $\det \begin{colorpmatrix} 8 & 4 \\ 4 & 4 \end{colorpmatrix} = 16 $ & $1$ & $23$, $41$, $3$ \\ \hline 
  
  & $\det:23$ & & \\ \hline

  $\mathfrak{S}(2;2, 0, 0)$ & $\det \begin{colorpmatrix} 8 & 2 \\ 2 & 4 \end{colorpmatrix} = 28 $ & $2$ & $27$, $25$, $0$ \\ \hline 
  
  $\mathfrak{S}(4;6, 1, 0)$ & $\det \begin{colorpmatrix} 8 & 6 \\ 6 & 8 \end{colorpmatrix} = 28 $ & $2$ & $19$, $53$, $4$ \\ \hline 
  
  \rowcolor{Gray} $\mathfrak{S}(1;0, 0, 0)$ & $\det \begin{colorpmatrix} 8 & 1 \\ 1 & 4 \end{colorpmatrix} = 31 $ & $3$ & $30$, $15$, $0$ \\ \hline 
  
  $\mathfrak{S}(5;8, 0, 1)$ & $\det \begin{colorpmatrix} 8 & 8 \\ 8 & 12 \end{colorpmatrix} = 32 $ & $2$ & $15$, $65$, $4$ \\ \hline 
  
  \rowcolor{Gray} $\mathfrak{S}(2;1, 0, 0)$ & $\det \begin{colorpmatrix} 8 & 3 \\ 3 & 6 \end{colorpmatrix} = 39 $ & $3$ & $24$, $33$, $0$ \\ \hline 
  
  \rowcolor{Gray} $\mathfrak{S}(3;4, 0, 0)$ & $\det \begin{colorpmatrix} 8 & 5 \\ 5 & 8 \end{colorpmatrix} = 39 $ & $3$ & $20$, $47$, $2$ \\ \hline 
  
  $\mathfrak{S}(5;11, 1, 0)^{\ast}$ & $\det \begin{colorpmatrix} 8 & 10 \\ 10 & 18 \end{colorpmatrix} = 44 $ & $3$ & $13$, $71$, $5$ \\ \hline 

  \rowcolor{Gray} $\mathfrak{S}(4;5, 1, 0)$ & $\det \begin{colorpmatrix} 8 & 7 \\ 7 & 12 \end{colorpmatrix} = 47 $ & $3$ & $16$, $59$, $2$ \\ \hline 
  
  \rowcolor{Gray} $\mathfrak{S}(5;7, 0, 1)$ & $\det \begin{colorpmatrix} 8 & 9 \\ 9 & 16 \end{colorpmatrix} = 47 $ & $3$ & $12$, $71$, $2$ \\ \hline 
  
  \rowcolor{Gray} $\mathfrak{S}(5;8, 2, 0)$ & $\det \begin{colorpmatrix} 8 & 9 \\ 9 & 16 \end{colorpmatrix} = 47 $ & $3$ & $14$, $67$, $4$ \\ \hline 
  
  $\mathfrak{S}(4;8, 0, 0)$ & $\det \begin{colorpmatrix} 8 & 8 \\ 8 & 14 \end{colorpmatrix} = 48 $ & $3$ & $15$, $63$, $3$ \\ \hline 
  
  $\mathfrak{S}(6;11, 1, 1)$ & $\det \begin{colorpmatrix} 8 & 12 \\ 12 & 24 \end{colorpmatrix} = 48 $ & $3$ & $9$, $81$, $3$ \\ \hline 
  
  $\mathfrak{S}(3;1, 1, 0)$ & $\det \begin{colorpmatrix} 8 & 4 \\ 4 & 8 \end{colorpmatrix} = 48 $ & $4$ & $21$, $41$, $0$ \\ \hline 
  
  \rowcolor{Gray} $\mathfrak{S}(5;10, 1, 0)$ & $\det \begin{colorpmatrix} 8 & 11 \\ 11 & 22 \end{colorpmatrix} = 55 $ & $3$ & $10$, $77$, $2$ \\ \hline 
  
  \rowcolor{Gray} $\mathfrak{S}(3;0, 1, 0)$ & $\det \begin{colorpmatrix} 8 & 5 \\ 5 & 10 \end{colorpmatrix} = 55 $ & $5$ & $18$, $49$, $0$ \\ \hline 
  
  $\mathfrak{S}(5;7, 2, 0)$ & $\det \begin{colorpmatrix} 8 & 10 \\ 10 & 20 \end{colorpmatrix} = 60 $ & $3$ & $11$, $73$, $1$ \\ \hline 
  
  $\mathfrak{S}(3;3, 0, 0)$ & $\det \begin{colorpmatrix} 8 & 6 \\ 6 & 12 \end{colorpmatrix} = 60 $ & $5$ & $17$, $53$, $1$ \\ \hline 
  
  $\mathfrak{S}(6;6, 5, 0)$ & $\det \begin{colorpmatrix} 8 & 10 \\ 10 & 20 \end{colorpmatrix} = 60 $ & $5$ & $13$, $69$, $5$ \\ \hline 
  
  $\mathfrak{S}(4;1, 0, 1)$ & $\det \begin{colorpmatrix} 8 & 6 \\ 6 & 12 \end{colorpmatrix} = 60 $ & $6$ & $15$, $57$, $0$ \\ \hline 
  
  \rowcolor{Gray} $\mathfrak{S}(4;7, 0, 0)$ & $\det \begin{colorpmatrix} 8 & 9 \\ 9 & 18 \end{colorpmatrix} = 63 $ & $3$ & $12$, $69$, $0$ \\ \hline 

\end{tabular}
\caption{Unirational families of fourfolds in $\mathcal{M}_{(2,2,2),7}$ described as the closure of 
the locus of fourfolds $[X]$ 
   containing some 
   smooth rational 
   surface $S=\mathfrak{S}(a;n_1,n_2,\ldots)\subset \mathbb{P}^7$  
   obtained as the image of the plane via the linear system 
   of curves of degree $a$ having $n_i$ general points of multiplicity $i$, for $i\geq 1$.
   All the surfaces are cut out by quadrics
   except those marked with an asterisk $({}^\ast)$.
   }
\label{Table: unirationality}
\end{table}

\begin{table}[htbp]
  \renewcommand{\arraystretch}{1.0}
  \centering
  \tabcolsep=1.5pt 
  \footnotesize
  \begin{tabular}{|c|c|c|c|}
  \hline
  Surface $S\subset X\subset \mathbb{P}^7$  &
  $\det \left(\begin{smallmatrix} H_X^4& H_X^2\cdot [S]\\ [S] \cdot H_X^2 & (S)_X^2 \end{smallmatrix}\right)$ &
  \begin{tabular}{c} Codim in \\ $\mathcal{M}_{(2,2,2),7}$ \end{tabular} &
  \begin{tabular}{c} Count of parameters \\ ($h^0(\mathcal I_{S/\mathbb{P}^7}(2))$,
       $h^0(N_{S/\mathbb{P}^7})$, 
      \\ $h^0(N_{S/X})$)\end{tabular} \\
  \hline
  \hline

  $\mathfrak{S}(6;8, 4, 0)$ & $\det \begin{colorpmatrix} 8 & 12 \\ 12 & 26 \end{colorpmatrix} = 64 $ & $3$ & $9$, $79$, $1$ \\ \hline 
  
  $\mathfrak{S}(4;4, 1, 0)$ & $\det \begin{colorpmatrix} 8 & 8 \\ 8 & 16 \end{colorpmatrix} = 64 $ & $4$ & $13$, $65$, $0$ \\ \hline 
  
  $\mathfrak{S}(6;0, 7, 0)$ & $\det \begin{colorpmatrix} 8 & 8 \\ 8 & 16 \end{colorpmatrix} = 64 $ & $5$ & $15$, $61$, $3$ \\ \hline 
  
  $\mathfrak{S}(2;0, 0, 0)$ & $\det \begin{colorpmatrix} 8 & 4 \\ 4 & 10 \end{colorpmatrix} = 64 $ & $6$ & $21$, $39$, $0$ \\ \hline 

  \rowcolor{Gray} $\mathfrak{S}(6;5, 5, 0)$ & $\det \begin{colorpmatrix} 8 & 11 \\ 11 & 24 \end{colorpmatrix} = 71 $ & $3$ & $10$, $75$, $0$ \\ \hline 
  
  \rowcolor{Gray} $\mathfrak{S}(7;7, 5, 1)$ & $\det \begin{colorpmatrix} 8 & 13 \\ 13 & 30 \end{colorpmatrix} = 71 $ & $3$ & $8$, $81$, $0$ \\ \hline 
  
  $\mathfrak{S}(6;2, 6, 0)$ & $\det \begin{colorpmatrix} 8 & 10 \\ 10 & 22 \end{colorpmatrix} = 76 $ & $4$ & $11$, $71$, $0$ \\ \hline 
  
  \rowcolor{Gray} $\mathfrak{S}(3;2, 0, 0)$ & $\det \begin{colorpmatrix} 8 & 7 \\ 7 & 16 \end{colorpmatrix} = 79 $ & $7$ & $14$, $59$, $0$ \\ \hline 
  
  $\mathfrak{S}(7;4, 6, 1)$ & $\det \begin{colorpmatrix} 8 & 12 \\ 12 & 28 \end{colorpmatrix} = 80 $ & $4$ & $9$, $77$, $0$ \\ \hline 
  
  \rowcolor{Gray} $\mathfrak{S}(7;4, 8, 0)$ & $\det \begin{colorpmatrix} 8 & 13 \\ 13 & 32 \end{colorpmatrix} = 87 $ & $5$ & $8$, $79$, $0$ \\ \hline 
  
  $\mathfrak{S}(8;4, 7, 2)$ & $\det \begin{colorpmatrix} 8 & 14 \\ 14 & 36 \end{colorpmatrix} = 92 $ & $6$ & $7$, $81$, $0$ \\ \hline

  & $\det:95$ & & \\ \hline

  $\mathfrak{S}(7;1, 9, 0)^{\ast}$ & $\det \begin{colorpmatrix} 8 & 12 \\ 12 & 30 \end{colorpmatrix} = 96 $ & $6$ & $9$, $75$, $0$ \\ \hline 

  \rowcolor{Gray} $\mathfrak{S}(8;1, 8, 2)$ & $\det \begin{colorpmatrix} 8 & 13 \\ 13 & 34 \end{colorpmatrix} = 103 $ & $7$ & $8$, $77$, $0$ \\ \hline 
  
  & $\det:108$ & & \\ \hline



  
\end{tabular}
\caption{Continuation of Table~\ref{Table: unirationality}.}
\label{Table: unirationality 2}
\end{table}

\newgeometry{left=1.0cm,bottom=1.0cm}

\thispagestyle{empty}

\begin{landscape}

\begin{table}[htbp]
 \renewcommand{\arraystretch}{1.5}
\centering
\tabcolsep=1.5pt 
\footnotesize
\begin{tabular}{|c|c|c|c||c|c|c|c|}
\hline
{\begin{minipage}[c]{0.30\textwidth} Surface $S\subset X\subset Y=Q_1\cap Q_2\subset\mathbb{P}^7$ \end{minipage}} &
$\left(\begin{smallmatrix} H_X^4& H_X^2\cdot [S]\\ [S] \cdot H_X^2 & (S)_X^2 \end{smallmatrix}\right)$ &
\begin{tabular}{c} Codim in \\ $\mathcal{M}_{(2,2,2),7}$ \end{tabular} &
{\begin{minipage}[c]{0.2\textwidth} Count of parameters \end{minipage}} & Fourfold $W$ &
Surface $U\subset W$  & 
\begin{minipage}[c]{0.20\textwidth} Associated Castelnuovo surface $\widetilde{U}$ \end{minipage} & $\left(\begin{smallmatrix} H_{\widetilde{U}}^2& H_{\widetilde{U}}\cdot K_{\widetilde{U}} \\  K_{\widetilde{U}} \cdot H_{\widetilde{U}} & K_{\widetilde{U}}^2 \end{smallmatrix}\right)$ \\
\hline
\hline

{\begin{minipage}[c]{0.30\textwidth} {Plane in $\mathbb{P}^7$}. This surface admits a congruence of $1$-secant lines inside $Y$.\end{minipage}} & \begin{tabular}{c} $\begin{pmatrix}
     8&1\\
     1&4
     \end{pmatrix}$ \\ $\det: 31$ \end{tabular} & $3$ & 
     \begin{tabular}{l} $h^0(\mathcal I_{S/\mathbb{P}^7}(2)) = 30$ \\
          $h^0(N_{S/\mathbb{P}^7}) = 15$ \\ $h^0(N_{S/X}) = 0$
     \end{tabular} & $W=\mathbb{P}^4$ & \begin{minipage}[c]{0.25\textwidth} {Smooth surface in $\mathbb{P}^4$ of degree $9$ and sectional genus $9$ cut out by one cubic and $3$ quartics.} \end{minipage} & $\widetilde{U}=U$ & \begin{tabular}{c} $\begin{pmatrix}
          9&7\\
          7&2
          \end{pmatrix}$ \\ $\det: -31$ \end{tabular} \\
\hline 
\multicolumn{8}{c}{
\begin{minipage}[t]{1.1\columnwidth}%
\emph{M2-command (1)}: \texttt{specialFourfold "plane in PP\^{}7"} \\
\emph{M2-command (2)}: \texttt{specialFourfold("prebuilt-example",1)}
\end{minipage}} \\     
\hline

{\begin{minipage}[c]{0.30\textwidth} {Smooth rational surface of degree $9$
     and sectional genus $3$, 
     obtained as the image of the plane via the linear system 
     of quintic curves having 
     $7$ general simple base points and one general triple point.
   This surface admits a congruence of $5$-secant cubic curves inside $Y$.} \end{minipage}} & 
   \begin{tabular}{c} $\begin{pmatrix}
          8&9\\
          9&16
          \end{pmatrix}$ \\ $\det: 47$ \end{tabular} & $3$ & 
          \begin{tabular}{l} $h^0(\mathcal I_{S/\mathbb{P}^7}(2)) = 12$ \\
               $h^0(N_{S/\mathbb{P}^7}) = 71$ \\ $h^0(N_{S/X}) = 2$
          \end{tabular} & \begin{minipage}[c]{0.20\textwidth} Smooth  complete intersection of two quadrics  in $\mathbb{P}^6$ \end{minipage} & \begin{minipage}[c]{0.25\textwidth} {Smooth surface in $\mathbb{P}^6$ of degree $16$ and sectional genus $14$ cut out by
               $2$ quadrics  and $9$ cubics.}\end{minipage} & \begin{minipage}[c]{0.20\textwidth} {$U$ is a simple internal projection of a smooth surface $\widetilde{U}\subset\mathbb{P}^7$} \end{minipage} & \begin{tabular}{c} $\begin{pmatrix}
               17&9\\
               9&2
               \end{pmatrix}$ \\ $\det: -47$ \end{tabular} \\
\hline 
\multicolumn{8}{c}{
\begin{minipage}[t]{1.1\columnwidth}%
\emph{M2-command (1)}: \texttt{specialFourfold surface \{5, 7, 0, 1\}} \\
\emph{M2-command (2)}: \texttt{specialFourfold("prebuilt-example",2)}
\end{minipage}} \\   
\hline 

{\begin{minipage}[c]{0.30\textwidth} {General internal projection of a general K3 surface of degree $14$ and genus $8$ in $\mathbb{P}^8$}. This surface admits a congruence of $9$-secant quintic curves inside~$Y$.\end{minipage}} & \begin{tabular}{c} $\begin{pmatrix}
          8&13\\
          13&28
          \end{pmatrix}$ \\ $\det: 55$ \end{tabular} & $3$ & 
          \begin{tabular}{l} $h^0(\mathcal I_{S/\mathbb{P}^7}(2)) = 9$ \\
               $h^0(N_{S/\mathbb{P}^7}) = 84$ \\ $h^0(N_{S/X}) = 6$
          \end{tabular} & \begin{minipage}[c]{0.20\textwidth} Cubic  fourfold  in $\mathcal C_{14}$\end{minipage} & \begin{minipage}[c]{0.25\textwidth} {Smooth surface in $\mathbb{P}^5$ of degree $13$ and sectional genus $12$ cut out by $7$ cubics}.   \end{minipage} & $\widetilde{U}=U$ & \begin{tabular}{c} $\begin{pmatrix}
               13&9\\
               9&2
               \end{pmatrix}$ \\ $\det: -55$ \end{tabular} \\
\hline 
\multicolumn{8}{c}{
\begin{minipage}[t]{1.1\columnwidth}%
\emph{M2-command (1)}: \texttt{specialFourfold "internal projection of K3 surface of genus 8"} \\
\emph{M2-command (2)}: \texttt{specialFourfold("prebuilt-example",3)}
\end{minipage}} \\     
\hline

{\begin{minipage}[c]{0.30\textwidth} {Smooth rational surface of degree $11$
     and sectional genus $5$, 
     obtained as the image of the plane via the linear system 
     of quintic curves having 
     $10$ general simple base points and one general double point.
   This surface admits a congruence of $7$-secant quartic curves inside $Y$.} \end{minipage}} & 
   \begin{tabular}{c} $\begin{pmatrix}
          8&11\\
          11&22
          \end{pmatrix}$ \\ $\det: 55$ \end{tabular} & $3$ & 
          \begin{tabular}{l} $h^0(\mathcal I_{S/\mathbb{P}^7}(2)) = 10$ \\
               $h^0(N_{S/\mathbb{P}^7}) = 77$ \\ $h^0(N_{S/X}) = 2$
          \end{tabular} & $W=\mathbb{P}^4$ & \begin{minipage}[c]{0.25\textwidth} {Surface in $\mathbb{P}^4$ of degree $12$ and sectional genus $12$ cut out by
               $7$ quintics and with $10$ singular points}\end{minipage} & \begin{minipage}[c]{0.20\textwidth} {$U$ is a simple internal projection of a smooth surface $\widetilde{U}\subset\mathbb{P}^5$ as in the previous row.} \end{minipage} & \begin{tabular}{c} $\begin{pmatrix}
               13&9\\
               9&2
               \end{pmatrix}$ \\ $\det: -55$ \end{tabular} \\
\hline 
\multicolumn{8}{c}{
\begin{minipage}[t]{1.1\columnwidth}%
\emph{M2-command (1)}: \texttt{specialFourfold surface \{5, 10, 1\}} \\
\emph{M2-command (2)}: \texttt{specialFourfold("prebuilt-example",4)}
\end{minipage}} \\   

\end{tabular}
\caption{Complete intersections of three quadrics $X\subset\mathbb{P}^7$ which are birational to rational fourfolds $W$.}
\label{Table: rational complete intersections of three quadrics in P7}
\end{table}

\newpage 
\thispagestyle{empty}

\begin{table}[htbp]
     \renewcommand{\arraystretch}{1.5}
    \centering
    \tabcolsep=1.5pt 
    \footnotesize
    \begin{tabular}{|c|c|c|c||c|c|c|c|}
    \hline
    {\begin{minipage}[c]{0.30\textwidth} Surface $S\subset X\subset Y=Q_1\cap Q_2\subset\mathbb{P}^7$ \end{minipage}} &
    $\left(\begin{smallmatrix} H_X^4& H_X^2\cdot [S]\\ [S] \cdot H_X^2 & (S)_X^2 \end{smallmatrix}\right)$ &
    \begin{tabular}{c} Codim in \\ $\mathcal{M}_{(2,2,2),7}$ \end{tabular} &
    {\begin{minipage}[c]{0.2\textwidth} Count of parameters \end{minipage}} & Fourfold $W$ &
    Surface $U\subset W$  & 
    \begin{minipage}[c]{0.20\textwidth} Associated Castelnuovo surface $\widetilde{U}$ \end{minipage} & $\left(\begin{smallmatrix} H_{\widetilde{U}}^2& H_{\widetilde{U}}\cdot K_{\widetilde{U}} \\  K_{\widetilde{U}} \cdot H_{\widetilde{U}} & K_{\widetilde{U}}^2 \end{smallmatrix}\right)$ \\
    \hline
    \hline

    {\begin{minipage}[c]{0.30\textwidth} {General external projection 
     of a smooth rational surface of degree $9$ and sectional genus $2$ in $\mathbb{P}^8$,
     obtained as the image of the plane via the linear system 
     of quartic curves having 
     $3$ general simple base points and one general double point.
   This surface admits a congruence of $9$-secant quintic curves inside $Y$.} \end{minipage}} & 
   \begin{tabular}{c} $\begin{pmatrix}
          8&9\\
          9&20
          \end{pmatrix}$ \\ $\det: 79$ \end{tabular} & $7$ & 
          \begin{tabular}{l} $h^0(\mathcal I_{S/\mathbb{P}^7}(2)) = 10$ \\
               $h^0(N_{S/\mathbb{P}^7}) = 71$ \\ $h^0(N_{S/X}) = 0$
          \end{tabular} & \begin{minipage}[c]{0.20\textwidth} Smooth Del Pezzo fourfold in $\mathbb{P}^7$ \end{minipage} & \begin{minipage}[c]{0.25\textwidth} {Singular surface in $\mathbb{P}^7$ of degree $21$ and sectional genus $17$ cut out by
               $5$ quadrics and $7$ cubics.}\end{minipage} & \begin{minipage}[c]{0.20\textwidth} {$U$ is a special external projection of a smooth surface $\widetilde{U}\subset\mathbb{P}^8$.} \end{minipage} & \begin{tabular}{c} $\begin{pmatrix}
               21 & 11 \\
               11 & 2
               \end{pmatrix}$ \\ $\det: -79$ \end{tabular} \\
\hline 
\multicolumn{8}{c}{
\begin{minipage}[t]{1.1\columnwidth}%
\emph{M2-command (1)}: \texttt{specialFourfold externalProjection surface \{4, 3, 1\}}  \\
\emph{M2-command (2)}: \texttt{specialFourfold("prebuilt-example",5)}
\end{minipage}} \\  
\hline

    {\begin{minipage}[c]{0.30\textwidth} {General nodal projection 
     of a smooth rational surface of degree $11$ and sectional genus $4$ in $\mathbb{P}^8$,
     obtained as the image of the plane via the linear system 
     of quintic curves having 
     $6$ general simple base points and $2$ general double points.
   This surface admits a congruence of $11$-secant sextic curves inside $Y$.} \end{minipage}} & 
   \begin{tabular}{c} $\begin{pmatrix}
          8&11\\
          11&26
          \end{pmatrix}$ \\ $\det: 87$ \end{tabular} & $5$ & 
          \begin{tabular}{l} $h^0(\mathcal I_{S/\mathbb{P}^7}(2)) = 9$ \\
               $h^0(N_{S/\mathbb{P}^7}) = 76$ \\ $h^0(N_{S/X}) = 0$
          \end{tabular} & \begin{minipage}[c]{0.20\textwidth} Smooth  complete intersection of two quadrics  in $\mathbb{P}^6$ \end{minipage} & \begin{minipage}[c]{0.25\textwidth} {Smooth surface in $\mathbb{P}^6$ of degree $17$ and sectional genus $15$ cut out by
               $2$ quadrics, $6$ cubics, and $2$ quartics.}\end{minipage} & $\widetilde{U} = U$ & \begin{tabular}{c} $\begin{pmatrix}
               17&11\\
               11&2
               \end{pmatrix}$ \\ $\det: -87$ \end{tabular} \\
\hline 
\multicolumn{8}{c}{
\begin{minipage}[t]{1.1\columnwidth}%
\emph{M2-command (1)}: \texttt{specialFourfold surface(\{5, 6, 2\}, NumNodes=>1)} \\
\emph{M2-command (2)}: \texttt{specialFourfold("prebuilt-example",6)}
\end{minipage}} \\   
\hline 

{\begin{minipage}[c]{0.30\textwidth} {Smooth rational surface of degree $12$ and sectional genus $6$,
     obtained as the image of the plane via the linear system 
     of septic curves having 
     one general simple base point and $9$ general double points.
   This surface admits a congruence of $7$-secant quartic curves inside $Y$. See also Remark~\ref{rmk: degeneration of congruence}.} \end{minipage}} & 
   \begin{tabular}{c} $\begin{pmatrix}
          8&12\\
          12&30
          \end{pmatrix}$ \\ $\det: 96$ \end{tabular} & $6$ & 
          \begin{tabular}{l} $h^0(\mathcal I_{S/\mathbb{P}^7}(2)) = 9$ \\
               $h^0(N_{S/\mathbb{P}^7}) = 75$ \\ $h^0(N_{S/X}) = 0$
          \end{tabular} & $W = \mathbb{P}^4$ & \begin{minipage}[c]{0.25\textwidth} {Surface in $\mathbb{P}^4$ of degree $12$ and sectional genus $12$ cut out by
               $6$ quintics and $4$ sextics and with $9$ singular points.}\end{minipage} & $\_\_$ & \begin{tabular}{c} $\begin{pmatrix}
               \_\_ & \_\_ \\
               \_\_ & \_\_
               \end{pmatrix}$ \\ $\det: \_\_$ \end{tabular} \\
\hline 
\multicolumn{8}{c}{
\begin{minipage}[t]{1.1\columnwidth}%
\emph{M2-command (1)}: \texttt{specialFourfold surface \{7, 1, 9\}} \\
\emph{M2-command (2)}: \texttt{specialFourfold("prebuilt-example",7)}
\end{minipage}} \\   

\end{tabular}
\caption{Continuation of Table~\ref{Table: rational complete intersections of three quadrics in P7}.}
\label{Table: continuation of Table 1}
\end{table}

\end{landscape}

\restoregeometry

\begin{rmk}\label{rmk: degeneration of congruence}{\rm 
  In all of the four cases listed in Table~\ref{Table: rational complete intersections of three quadrics in P7},
  the surface $S\subset X\subset \mathbb{P}^7$ is cut out by quadrics.
  The same holds true for the surface $S$ in the second row of Table~\ref{Table: continuation of Table 1}.
  
  The surface $S\subset\mathbb{P}^7$ as in the first row of 
  Table~\ref{Table: continuation of Table 1} is cut out by $10$ quadrics and one cubic. The $10$ quadrics through $S$ 
  define the union of $S$ with a line.
  
  In the case of the third row of Table~\ref{Table: continuation of Table 1},
  the surface $S\subset\mathbb{P}^7$ is cut out by $9$ quadrics and one cubic.
  These $9$ quadrics through $S$ define a reducible surface $S\cup P$ consisting of the union of $S$ with a plane $P$, and
  which is a degeneration of the surface appearing in the third row of Table~\ref{Table: rational complete intersections of three quadrics in P7} (that is, an internal projection of a K3 surface of genus $8$).
  }
  \end{rmk}

\section{Computations with \emph{Macaulay2}}

Here are some tips of how to use functions from the \emph{Macaulay2} package \emph{SpecialFanoFourfolds} \cite{SpecialFanoFourfoldsSource,macaulay2}
 to construct and work with the examples in Tables \ref{Table: rational complete intersections of three quadrics in P7} and \ref{Table: continuation of Table 1}.
For instance, we now construct a random example as in the second row of Table \ref{Table: rational complete intersections of three quadrics in P7} defined over the finite field $\mathbb{F}_{33331}$.
    {\footnotesize
    \begin{Verbatim}[commandchars=&!$]
&colore!darkorange$!M2 -q --no-preload$
&colore!output$!Macaulay2, version 1.22$
    &colore!darkorange$!i1 :$ &colore!airforceblue$!needsPackage$ "&colore!bleudefrance$!SpecialFanoFourfolds$"; &colore!commentoutput$!-- version 2.7.1$
    &colore!darkorange$!i2 :$ S = &colore!bleudefrance$!surface$({5, 7, 0, 1}, &colore!airforceblue$!ZZ$/33331);
    &colore!circOut$!o2 :$ &colore!output$!ProjectiveVariety, surface in PP^7$
    &colore!darkorange$!i3 :$ X = &colore!bleudefrance$!specialFourfold$ S; &colore!commentoutput$!-- shortcut for specialFourfold(S,random({3:{2}},S))$
    &colore!circOut$!o3 :$ &colore!output$!ProjectiveVariety, complete intersection of three quadrics in PP^7$
    &colore!output$!     containing a surface of degree 9 and sectional genus 3$
    &colore!darkorange$!i4 :$ &colore!bleudefrance$!describe$ X
    &colore!circOut$!o4 =$ &colore!output$!Complete intersection of 3 quadrics in PP^7$
         &colore!output$!of discriminant 47 = det| 8 9  |$
         &colore!output$!                        | 9 16 |$
         &colore!output$!containing a smooth surface of degree 9 and sectional genus 3$
         &colore!output$!cut out by 12 hypersurfaces of degree 2$
    &colore!darkorange$!i5 :$ &colore!bleudefrance$!surface$ X
    &colore!circOut$!o5 =$ &colore!output$!S$
    &colore!circOut$!o5 :$ &colore!output$!ProjectiveVariety, surface in PP^8$
    &colore!darkorange$!i6 :$ Y = &colore!bleudefrance$!ambientFivefold$ X; &colore!commentoutput$!-- complete intersection of 2 quadrics through X$
    &colore!circOut$!o6 :$ &colore!output$!ProjectiveVariety, 5-dimensional subvariety of PP^7$
    \end{Verbatim}
    } \noindent
Notice that the command \texttt{specialFourfold("prebuilt-example",2)}
returns a similar object as above, but with some data   
cached in memory to make calculations run more quickly. 

Below we take the congruence of $9$-secant quintic curves to the surface $S$
inside the ambient fivefold $Y$ of $X$. Next, 
 we perform a count of parameters similar to that reported 
 in the fourth column of the tables (see Proposition~\ref{conti_parametri}).
    {\footnotesize
    \begin{Verbatim}[commandchars=&!$]
    &colore!darkorange$!i7 :$ &colore!bleudefrance$!detectCongruence$(X, 3);
    &colore!circOut$!o7 :$ &colore!output$!Congruence of 5-secant cubic curves to S in Y$
    &colore!darkorange$!i8 :$ &colore!bleudefrance$!parameterCount$(X, &colore!airforceblue$!Verbose$=>true) &colore!commentoutput$!-- some output lines are omitted$
    &colore!output$!-- h^1(N_{S,Y}) = 0$
    &colore!output$!-- h^0(N_{S,Y}) = 23$
    &colore!output$!-- h^1(O_S(2)) = 0 and h^0(I_{S,Y}(2)) = 10 = h^0(O_Y(2)) - \chi(O_S(2))$
    &colore!output$!-- h^0(N_{S,X}) = 2$
    &colore!output$!-- codim{[X] : S \subset X \subset Y} <= 3$
    &colore!circOut$!o8 =$ &colore!output$!(3, (10, 23, 2))$
    \end{Verbatim}
    } \noindent 
We now compute the Castelnuovo surface $\widetilde{U}$ associated with $X$
(this can take a while for newly created objects and therefore with empty cache).
    {\footnotesize
    \begin{Verbatim}[commandchars=&!$]
    &colore!darkorange$!i9 :$ Utilde = &colore!bleudefrance$!associatedCastelnuovoSurface$(X, &colore!airforceblue$!Verbose$=>true); &colore!commentoutput$!-- many output lines are omitted$
    &colore!output$!-- computed the map mu from the fivefold in PP^7 to PP^6 defined by the hypersurfaces$
    &colore!output$!   of degree 5 with points of multiplicity 3 along the surface S of degree 9 and genus 3$
    &colore!output$!-- computing the surface U corresponding to the fourfold X$
    &colore!circOut$!o9 :$ &colore!output$!ProjectiveVariety, Castelnuovo surface associated to X$
    &colore!darkorange$!i10 :$ &colore!airforceblue$!? ideal$ Utilde
    &colore!circOut$!o10 =$ &colore!output$!surface of degree 17 and sectional genus 14 in PP^7$
    &colore!output$!      cut out by 7 hypersurfaces of degree 2$
    \end{Verbatim}
    } \noindent
From the last output we can recover some data obtained in the construction. For instance, the Fano map $\mu:Y\dashrightarrow W\subset\mathbb{P}^6$ 
defined by the quintic hypersurfaces with points of multiplicity $3$ along the surface $S$ 
can be recovered as shown below.
We finally restrict this map $\mu$ to $X$ and compose it with the inverse of a rational parametrization of $W$, 
thus obtaining a birational map $X\stackrel{\simeq}{\dashrightarrow}\mathbb{P}^4$.
    {\footnotesize
    \begin{Verbatim}[commandchars=&!$]
    &colore!darkorange$!i11 :$ mu = &colore!airforceblue$!first$ &colore!bleudefrance$!building$ Utilde;
    &colore!circOut$!o11 =$ &colore!output$!MultirationalMap (dominant rational map from Y to 4-dimensional subvariety of PP^6)$
    &colore!darkorange$!i12 :$ (mu|X) * (&colore!airforceblue$!parametrize target$ mu)^-1;
    &colore!circOut$!o12 =$ &colore!output$!MultirationalMap (rational map from X to PP^4)$
    \end{Verbatim}
    } \noindent

\providecommand{\bysame}{\leavevmode\hbox to3em{\hrulefill}\thinspace}
\providecommand{\MR}{\relax\ifhmode\unskip\space\fi MR }
\providecommand{\MRhref}[2]{%
  \href{http://www.ams.org/mathscinet-getitem?mr=#1}{#2}
}
\providecommand{\href}[2]{#2}

\end{document}